\documentclass[11pt, reqno]{amsart}

\usepackage{latexsym}
\usepackage{amssymb}
\usepackage{mathrsfs}
\usepackage{amsmath}
\usepackage{fancybox,color}
\usepackage{enumerate}
\usepackage{hyperref}
\usepackage[latin1]{inputenc}
\usepackage{eurosym}
	\usepackage{cases}
\usepackage{color}

\newcommand{\kom}[1]{}
\renewcommand{\kom}[1]{{\bf [#1]}}

 \def\1{\raisebox{2pt}{\rm{$\chi$}}}
\def\a{{\bf a}}

\newtheorem{theorem}{Theorem}[section]

\newtheorem{lemma}[theorem]{Lemma}

\newtheorem{definition}[theorem]{Definition}
\newtheorem{remark}[theorem]{Remark}
\newtheorem{example}[theorem]{Example}

\numberwithin{equation}{section}

\newcommand{\R}{{\mathbb R}}

\newcommand{\N}{{\mathbb N}}

 \newcommand{\eps}{{\varepsilon}}
 \def\1{\raisebox{2pt}{\rm{$\chi$}}}
 
\newcommand{\abs}[1]{\left|#1\right|}
\newcommand{\norm}[1]{\left|\left|#1\right|\right|}
\newcommand{\Rn}{\mathbb{R}^n}
\newcommand{\osc}{\operatorname{osc}}

\def\vint_#1{\mathchoice%
          {\mathop{\kern 0.2em\vrule width 0.6em height 0.69678ex depth -0.58065ex
                  \kern -0.8em \intop}\nolimits_{\kern -0.4em#1}}%
          {\mathop{\kern 0.1em\vrule width 0.5em height 0.69678ex depth -0.60387ex
                  \kern -0.6em \intop}\nolimits_{#1}}%
          {\mathop{\kern 0.1em\vrule width 0.5em height 0.69678ex
              depth -0.60387ex
                  \kern -0.6em \intop}\nolimits_{#1}}%
          {\mathop{\kern 0.1em\vrule width 0.5em height 0.69678ex depth -0.60387ex
                  \kern -0.6em \intop}\nolimits_{#1}}}
\def\vintslides_#1{\mathchoice%
          {\mathop{\kern 0.1em\vrule width 0.5em height 0.697ex depth -0.581ex
                  \kern -0.6em \intop}\nolimits_{\kern -0.4em#1}}%
          {\mathop{\kern 0.1em\vrule width 0.3em height 0.697ex depth -0.604ex
                  \kern -0.4em \intop}\nolimits_{#1}}%
          {\mathop{\kern 0.1em\vrule width 0.3em height 0.697ex depth -0.604ex
                  \kern -0.4em \intop}\nolimits_{#1}}%
          {\mathop{\kern 0.1em\vrule width 0.3em height 0.697ex depth -0.604ex
                  \kern -0.4em \intop}\nolimits_{#1}}}

\newcommand{\aveint}[2]{\mathchoice%
          {\mathop{\kern 0.2em\vrule width 0.6em height 0.69678ex depth -0.58065ex
                  \kern -0.8em \intop}\nolimits_{\kern -0.45em#1}^{#2}}%
          {\mathop{\kern 0.1em\vrule width 0.5em height 0.69678ex depth -0.60387ex
                  \kern -0.6em \intop}\nolimits_{#1}^{#2}}%
          {\mathop{\kern 0.1em\vrule width 0.5em height 0.69678ex depth -0.60387ex
                  \kern -0.6em \intop}\nolimits_{#1}^{#2}}%
          {\mathop{\kern 0.1em\vrule width 0.5em height 0.69678ex depth -0.60387ex
                  \kern -0.6em \intop}\nolimits_{#1}^{#2}}}

\newcommand{\half}{{\frac{1}{2}}}

\newcommand{\ol}{\overline}
\newcommand{\Om}{\Omega}

\newcommand{\dist}{\operatorname{dist}}

\newcommand{\diam}{\operatorname{diam}}

\newcommand{\tr}{\operatorname{tr}}
\renewcommand{\a}{\alpha}

\begin{document}

\title[Regularity for normalized $p$-Laplacian]{ $C^{1,\a}$ regularity for the normalized $p$-Poisson problem}

\author[Attouchi]{Amal Attouchi}
\address{Department of Mathematics and Statistics, University of
Jyv\"askyl\"a, PO~Box~35, FI-40014 Jyv\"askyl\"a, Finland}
\email{amalattouchi@gmail.com}

\author[Parviainen]{Mikko Parviainen}
\email{mikko.j.parviainen@jyu.fi}

\author[Ruosteenoja]{Eero Ruosteenoja}
\email{eero.ruosteenoja@jyu.fi}

\date{\today}

\begin{abstract}
We consider the normalized $p$-Poisson problem
$$-\Delta^N_p u=f \qquad \text{in}\quad \Om.$$
The normalized $p$-Laplacian $\Delta_p^{N}u:=|D u|^{2-p}\Delta_p u$ is in non-divergence form and arises for example from stochastic games. We prove $C^{1,\alpha}_{\text{\text{loc}}}$ regularity with nearly optimal $\alpha$ for viscosity solutions of this problem. In the case $f\in L^{\infty}\cap C$ and $p>1$ we use methods both from viscosity and weak theory, whereas in the case $f\in L^q\cap C$, $q>\max(n,\frac p2,2)$, and $p>2$ we rely on the tools of nonlinear potential theory.   
\end{abstract}
\maketitle

\section{Introduction}
\label{sec:intro}

In this paper we study local regularity properties of the inhomogeneous normalized $p$-Laplace equation
\begin{equation}
\label{normpl}
-\Delta_p^{N}u=f \qquad \text{in}\quad \Om\subset \R^n.
\end{equation}
The normalized $p$-Laplacian is defined as 
\[
\Delta_p^{N}u:=|D u|^{2-p}\Delta_p u=\Delta u +(p-2)\Delta_{\infty}^N u,
\]
where $\Delta_{\infty}^N u:=\langle D^2u\frac{D u}{\abs{D u}}, \frac{D u}{\abs{D u}}\rangle$ denotes the normalized infinity Laplacian. The motivation to study these types of normalized operators stems partially from their connections to stochastic games and their applications to image processing. The normalized $p$-Laplacian is gradient dependent and discontinuous, so we cannot directly rely on the existing general $C^{1,\alpha}$ regularity theory of viscosity solutions. Only H\"older continuity for solutions of \eqref{normpl} follows from the regularity theory for uniformly elliptic equations, see \cite{caffarelli89,caffarellicabrebook}.

Our aim is to show local H\"older continuity for gradients of viscosity solutions of \eqref{normpl} by relying on different methods depending on regularity assumptions of the source term $f$. Assuming first that $f\in L^{\infty}(\Omega)\cap C(\Omega)$, we show that solutions of \eqref{normpl} for $p>1$ are of class $C^{1,\alpha}_{\text{loc}}$ for some  $\alpha>0$ depending on $p$ and the dimension $n$. 

\begin{theorem}\label{thm:main1}
Assume that $p>1$ and $f\in L^{\infty}(\Omega)\cap C(\Omega)$.  There exists $\alpha=\alpha(p,n)>0$ such that any viscosity solution $u$ of
  \eqref{normpl} is in $C^{1,\alpha}_{\text{\emph{loc}}}(\Omega)$, and for any $\Omega' \subset \subset \Omega$,
$$
[u]_{C^{1,\alpha}(\Omega')} \le C=C \left(p,n,d,d',||u||_{L^\infty(\Omega)},\norm{f}_{L^\infty(\Omega)} \right),
$$
where $d=\diam(\Om)$ and $d'=\dist(\Om', \partial\Om)$.
\end{theorem}
The idea is to first show an improvement of flatness for a slightly modified version of equation \eqref{normpl}, and then proceed by iteration. Earlier, in the restricted case $p\geq 2$, a $C^2$ domain $\Om$ and $f\in C(\overline{\Om})$, Birindelli and Demengel \cite[Proposition 3.5]{birdemen2} proved global H\"older continuity for the gradient of viscosity solutions of \eqref{normpl} by studying eigenvalue problems related to the $p$-Laplacian. In the case $p\geq 2$ we provide an alternative proof by showing first that viscosity solutions of \eqref{normpl} are weak solutions of
\begin{equation}\label{usualpl}
-\Delta_p u=|D u|^{p-2} f\qquad\text{in}\quad\Omega,
\end{equation}
and then relying on the known regularity results for quasilinear PDEs to see that weak solutions of \eqref{usualpl} are locally of class $C^{1,\alpha}$. 

Restricting to the case $p>2$, we can relax the estimate of Theorem \ref{thm:main1} by providing a control on the H\"older estimate of the gradient that depends on a weaker norm of $f$.

\begin{theorem}\label{thm:main2} Assume that $p>2$, $q>\max(2,n, p/2)$, $f\in C(\Om)\cap L^q(\Om)$. Then any  viscosity solution $u$ of \eqref{normpl}  is in $C^{1,\alpha}_{\text{\emph{loc}}}(\Omega)$ for some
$\alpha =\alpha(p, q,n)$. Moreover, for any $\Om''\subset \subset\Omega' \subset \subset \Omega$, with $\Om'$ smooth enough, we have
\[
[u]_{C^{1,\alpha}(\Omega'')}\leq C=C\left(p,q,n,d,d'',||u||_{L^\infty(\Om)},\norm{f}_{L^q(\Om)}\right),
\]
where $d=\text{\emph{diam}}\, (\Om)$ and $d''=\text{\emph{dist}}\, (\Om'',\partial \Om')$.
\end{theorem}

The main idea to prove Theorem \ref{thm:main2} relies on approximations, where we use classical methods from the weak theory and potential estimates developed by Duzaar and Mingione in \cite{DM2010}. In the proof we also show that under the assumptions of Theorem \ref{thm:main2}, there exists a weak solution of equation \eqref{usualpl} which is in $C^{1,\alpha}_{\text{loc}}(\Omega)$.

It is well known that $p$-harmonic functions are of class $C^{1,\a_0}_{\text{loc}}$ for some
maximal exponent $ 0 < \a_0 < 1$ that depends only upon $n$ and $p$. This was shown independently by Uraltseva \cite{uraltseva68} and Uhlenbeck \cite{uhlenbeck77} in the case $p>2$, and later extended to the case $p>1$, see \cite{Dib83,lewis83} and also \cite{manfredi86,iwaniecm89} for related research. The question of optimal regularity for $p$-Laplace equations in divergence form has attracted a lot of attention recently, see Section \ref{chpt6} for further references. Since the solutions of \eqref{normpl}
should not be expected to be more regular than $p$-harmonic functions, the maximal exponent
$\a_0$ is a natural upper bound for $C^{1,\a}$ regularity
for equation \eqref{normpl}. In the following theorem we obtain nearly optimal regularity for solutions of \eqref{normpl}. 

\begin{theorem}\label{thm:main3}
Fix an  arbitrary $\xi\in(0,\alpha_0)$, where $\alpha_0$ is the optimal H\"older exponent for gradients of $p$-harmonic functions in terms of an a priori estimate.

If $p>1$ and $f\in L^{\infty}(\Omega)\cap C(\Omega)$, then viscosity solutions to \eqref{normpl} are in $C^{1,\alpha_0-\xi}_{\text{\emph{loc}}}(\Omega)$.

If $p>2$, $q>\max(2,n, p/2)$ and  $f\in C(\Om)\cap L^q(\Om)$, then viscosity solutions to \eqref{normpl} are in $C^{1,\alpha_\xi}_{\text{\emph{loc}}}(\Omega)$, where $\alpha_\xi:=\min(\alpha_0-\xi,1-n/q)$. Moreover the estimates given in the previous theorems hold for $\alpha_\xi$.
\end{theorem}

When the gradient is sufficiently large, the result follows from the classical regularity results for uniformly elliptic equations. When the gradient is small, the first step is to use local $C^{1,\alpha}$ regularity of the solutions of \eqref{normpl}, proved in Theorems \ref{thm:main1} and \ref{thm:main2}, to show that the solutions can be approximated by $p$-harmonic functions in $C^{1,\alpha}$. The next step is to use suitable rescaled functions and iteration to obtain the required oscillation estimate.

Over the last decade, equation \eqref{normpl} and similar normalized equations have received growing attention, partly due to the stochastic zero-sum \emph{tug-of-war} games defined by Peres, Schramm, Sheffield and Wilson in \cite{peress08, peresssw09}. In \cite{peress08} Peres and Sheffield studied a connection between equation \eqref{normpl} and the game tug-of-war with noise and running pay-off. The game-theoretic interpretation led to new regularity proofs in the case $f=0$ in \cite{luirops13}, and later in the case of bounded and positive $f$ in \cite{ruosteenoja16}, see also \cite{charroddm13} for a PDE approach. Regularity studies were extended to the parabolic version $u_t=\Delta^N_p u$ in \cite{manfredipr10c,banerjee2015dirichlet,jins15} and led to applications in image processing, see e.g.\, \cite{does11, MR3416134}.

This paper is organized as follows. In Section \ref{prel} we fix the notation and gather some definitions and tools which we need later. In Section \ref{chpt3} we give two proofs for Theorem \ref{thm:main1}, in Section \ref{chpt4} we prove Theorem \ref{thm:main2}, and in Section \ref{chpt6} Theorem \ref{thm:main3}.

\noindent \textbf{Acknowledgements.} MP is supported by the Academy of Finland and ER is supported by the Vilho, Kalle and Yrj\"o V\"ais\"al\"a foundation. The authors would like to thank Peter Lindqvist and Giuseppe Mingione for useful discussions.

\section{Preliminaries}\label{prel}
Throughout the paper $\Omega\subset \R^n$ is a bounded domain. We use the notation
\[
\vint_{A} u\, \mathrm{d}x:=\frac{1}{|A|}\int_{A} u\, \mathrm{d}x
\]
for the mean value of a function $u$ in a measurable set $A\subset \Om$ with Lebesgue measure $|A|>0$.

For $p>1$, we denote by $\Lambda$ and $\lambda$ the ellipticity constants of the normalized $p$-Laplacian $\Delta^N_p$. Recalling the expression 
\[
\Delta^N_p u=\Delta u +(p-2)\Delta_{\infty}^N u=\tr\left((I+(p-2)\frac{D u\otimes D u}{|D u|^2})D^2 u\right)
\]
and calculating for arbitrary $\eta\in \R^n$, $|\eta|=1$,
\begin{align*}
\langle (I+(p-2)\frac{D u\otimes D u}{|D u|^2})\eta,\eta\rangle &=|\eta|^2+(p-2)\frac{\langle \eta, D u\rangle^2}{|D u|^2}\\
& =1+(p-2)\frac{\langle \eta, D u\rangle^2}{|D u|^2},
\end{align*}
we see that $\Lambda=\max (p-1, 1)$ and $\lambda=\min (p-1, 1)$.

We denote by $S^n$ the set of symmetric $n\times n$ matrices. For $a,b\in \R^n$, we denote by $a\otimes b$ the $n\times n$-matrix for which $(a\otimes b)_{ij}=a_i b_j$.

We will use the \emph{Pucci operators}
\[
P^+(X):=\sup_{A\in \mathcal{A}_{\lambda,\Lambda}} -\tr(AX)
\]
and
\[
P^-(X):=\inf_{A\in \mathcal{A}_{\lambda,\Lambda}} -\tr(AX),
\]
where $\mathcal{A}_{\lambda,\Lambda}\subset S^n$ is a set of symmetric $n\times n$ matrices whose eigenvalues belong to $[\lambda,\Lambda]$.

When studying H\"older and $C^{1,\alpha}$ regularity, for $\alpha\in (0,1]$ and a ball $B_r\subset \R^n$ we use the notation
\[
[u]_{C^{0,\alpha}(B_r)}:=\sup_{x,y\in B_r,x\neq y} \frac{|u(x)-u(y)|}{|x-y|^\alpha}
\]
for H\"older continuous functions, and
\[
[u]_{C^{1,\alpha}(B_r)}:=[u]_{C^1(B_r)}+\sup_{x,y\in B_r,x\neq y} \frac{|D u(x)-D u(y)|}{|x-y|^\alpha}
\]
for functions of class $C^{1,\alpha}$. Here $[u]_{C^1(B_r)}:=\sup_{x\in B_r}|D u(x)|$.

Recall that weak solutions to $-\Delta_p u:=-\operatorname{div}(\abs{D u}^{p-2}D u)=0$ are called $p$-harmonic functions. We will use the known $C^{1,\alpha_0}_{\text{\text{loc}}}$ a priori estimate in Sections \ref{chpt4} and \ref{chpt6}. The existence of the optimal $\alpha_0=\alpha_0(p,n)$ follows from the known regularity estimates for the homogeneous $p$-Laplace equation.

The normalized $p$-Laplacian is undefined when $D u=0$,
where  it has  a bounded discontinuity. This can be remediated adapting
the notion of viscosity solution using the upper and
lower semicontinuous envelopes (relaxations) of the operator, see \cite{crandall1992user}.

\begin{definition}
Let $\Om$ be a bounded domain and $1<p<\infty$. An upper semicontinuous function $u$ is  a viscosity subsolution of \eqref{normpl}
if for all $x_0\in\Om$  and $\phi\in C^2(\Om)$ such that $u-\phi$ attains a local maximum at $x_0$, one has
\begin{equation*}
\begin{cases}
-\Delta_p^N \phi(x_0)\leq f(x_0), &\text{if}\quad D \phi(x_0)\neq 0,\\
-\Delta\phi(x_0)-(p-2)\lambda_{max} (D^2\phi(x_0))\leq f(x_0),\, &\text{if}\ D \phi(x_0)=0\,\,\text{and}\,\, p\geq 2,\\
-\Delta\phi(x_0)-(p-2)\lambda_{min} (D^2\phi(x_0))\leq f(x_0),\, &\text{if}\ D \phi(x_0)=0\,\,\text{and}\,\, 1<p<2.
\end{cases}
\end{equation*}
 A lower semicontinuous function $u$ is a viscosity supersolution of \eqref{normpl}
if for all $x_0\in\Om$  and $\phi\in C^2(\Om)$ such that $u-\phi$ attains a local minimum at $x_0$, one has
\begin{equation*}
\begin{cases}
-\Delta_p^N \phi(x_0)\geq f(x_0), &\text{if}\quad D \phi(x_0)\neq 0,\\
-\Delta\phi(x_0)-(p-2)\lambda_{min} (D^2\phi(x_0))\geq f(x_0),\, &\text{if}\ D \phi(x_0)=0\,\text{and}\,\, p\geq 2,\\
-\Delta\phi(x_0)-(p-2)\lambda_{max} (D^2\phi(x_0))\geq f(x_0),\, &\text{if}\ D \phi(x_0)=0\,\text{and}\, 1<p<2.
\end{cases}
\end{equation*}
We say that $u$ is a viscosity solution of \eqref{normpl} in $\Om$ if it is both a viscosity sub- and
supersolution.
\end{definition}

We will make use of the equivalence between weak and viscosity solutions to the $p$-Laplace equation $\Delta_p u=0$. This was first proved in \cite{juutinenlm01} by using the full uniqueness machinery of the theory of viscosity solutions, and later in \cite{julin2012new} without relying on the uniqueness. The techniques of the second paper are particularly important for us in Section \ref{section3.2}, where we do not have uniqueness.

\section{Two proofs for Theorem \ref{thm:main1}}\label{chpt3}
In this section we give two proofs for Theorem \ref{thm:main1}. In the first subsection we use an iteration method often used to show $C^{1,\alpha}$ regularity for elliptic equations. Recently, Imbert and Silvestre \cite{silimb} used this method to show $C^{1,\alpha}$ regularity for viscosity solutions of $|D u|^\gamma F(D^2 u)=f$, where $F$ is uniformly elliptic.

In Section \ref{section3.2} we give another proof for Theorem \ref{thm:main1} in the case $p\geq 2$ by showing that a viscosity solution to \eqref{normpl} is also a weak solution to \eqref{usualpl}. 

\subsection{First proof by improvement of flatness and iteration}
In this subsection we give a first proof for Theorem \ref{thm:main1}. We assume that $p>1$ and $f\in L^{\infty}(\Omega)\cap C(\Omega)$,
and we want to show that there exists $\alpha=\alpha(p,n)>0$ such that any viscosity solution $u$ of
  \eqref{normpl} is in $C^{1,\alpha}_{\text{loc}}(\Omega)$, and for any $\Omega' \subset \subset \Omega$,
$$
[u]_{C^{1,\alpha}(\Omega')} \le C=C \left(p,n,d,d',||u||_{L^\infty(\Omega)},\norm{f}_{L^\infty(\Omega)} \right),
$$
where $d=\diam(\Om)$ and $d'=\dist(\Om', \partial \Om)$.

Since H\"older continuous functions can be characterized by the rate of their approximations by polynomials (see \cite{krylovholder}), it is sufficient to prove that there exists some constant $C$ such that for all $x\in\Om$ and  $r\in(0,1)$, there exists $q=q(r,x)\in\R^n$ for which
$$\underset{B_r(x)}{\osc} (u(y)-u(x)-q\cdot (x-y))\leq Cr^{1+\a}.$$
If one also starts with a solution $u$ such that $\osc{u}\leq 1$, then it is sufficient to choose a suitable $\rho\in(0,1)$ such that the previous inequality holds true for $r=r_k=\rho^k$,  $q=q_k$  and $C=1$ by proceeding by induction on $k\in\N$. The balls $B_r(x)$ for $x\in \Omega$ and $r<\text{dist}\, (x,\partial \Omega)$ covering the domain $\Omega$, we may work on balls. By translation, it is enough to show that the solution is  $C^{1,\a}$ at 0, and by considering 
\[
u_r(y)=r^{-2}u(x+ry),
\]
we may work on the unit ball $B_1(0)$. Finally, considering $u-u(0)$ if necessary, we may suppose that $u(0)=0$.

The idea of the proof is to first study the deviations of  $u$ from planes $w(x)=u(x)-q\cdot x$  which satisfy
 \begin{equation}\label{rescpb}
-\Delta w-(p-2) \left\langle D^2w\frac{D w+q}{\abs{D w+q}}, \frac{D w+q}{\abs{D w+q}}\right\rangle=f\quad\text{in}\quad B_1
\end{equation}
in the viscosity sense, and show equicontinuity for uniformly bounded solutions in Lemma \ref{compactres}. By the Arzel\`a-Ascoli theorem we get compactness, which, together with Lemma \ref{lip2}, we use to show improvement of flatness for solutions of \eqref{rescpb} in Lemma \ref{flatle}. Finally, we prove $C^{1,\alpha}$ regularity for solutions of \eqref{normpl} in Lemma \ref{lemiter} by using Lemma \ref{flatle} and iteration.

For the proofs of Lemmas \ref{compactres} and \ref{flatle} we reduce the problem by rescaling. Let $\kappa=(2||u||_{L^\infty(B_1)}+\eps_0^{-1}||f||_{L^\infty(B_1)})^{-1}$. Setting $\tilde{u}=\kappa u$, then $\tilde{u}$ satisfies
\begin{equation*}
-\Delta_p^N\left(\tilde{u}\right)=\tilde{f}
\end{equation*}
with $||\tilde{u}||_{L^\infty(B_1)}\leq \frac{1}{2}$ and $||\tilde{f}||_{L^\infty(B_1)}\le \eps_0$. Hence, without loss of generality we may assume in Theorem \ref{thm:main1} that $||u||_{L^\infty(B_1)}\leq 1/2$ and $||f||_{L^\infty(B_1)}\leq \eps_0$, where $\eps_0=\eps_0(p,n)$ is chosen later.

In order to prove Theorem \ref{thm:main1}, we will first need the following equicontinuity lemma for viscosity solutions to equation \eqref{rescpb}.

\begin{lemma}\label{compactres}
For all $r\in (0,1)$, there exist $\beta=\beta(p,n)\in(0,1)$ and $C=C(p,n,r,\osc_{B_1}(w),\norm{f}_{L^n(B_1)})>0$ such that any viscosity solution $w$ of \eqref{rescpb} satisfies 
\begin{equation}
[w]_{C^{0,\beta}(B_r)}\leq C.
\end{equation}
\end{lemma}

\begin{proof}
Equation \eqref{rescpb}  can be rewritten as $$-\tr{\left( \left(I+(p-2)\dfrac{D w+q}{|D w+q|}\otimes\dfrac{D w+q}{|D w+q|}\right) D^2 w\right)}=f.$$
Recalling the definitions of the Pucci operators $P^+$ and $P^-$ respectively, we have
$$\left\{\begin{array}{ll}P^+(D^2w)+|f|\geq 0\\
P^-(D^2w)-|f|\leq 0.\end{array}\right.$$
By the classical result of Caffarelli in \cite[Proposition 4.10]{caffarelli89}, there exists $\beta=\beta(p,n)\in(0,1)$ such that
\[
[w]_{C^{0,\beta}(B_r)}\leq  C=C\left(p,n,r,\underset{B_1}{\osc(w)},\norm{f}_{L^n(B_1)}\right).\qedhere
\] 
\end{proof}

The next lemma is needed to prove the key Lemma \ref{flatle}, where we show improvement of flatness. 
For convenience, we postpone the technical proof of Lemma \ref{lip2} and present it at the end of this section.

\begin{lemma}\label{lip2}
	 Assume that $f\equiv0$ and let $w$ be a viscosity solution to equation \eqref{rescpb} with $\underset{B_1}{\osc{w}}\leq 1$. For all $r\in (0,\frac12]$, there exist constants $ C_0=C_0(p,n)>0$  and $\beta_1=\beta_1(p,n)>0$ such that 
	\begin{equation}
[w]_{C^{1,\beta_1}(B_{r})}\leq  C_0.\qedhere
	\end{equation}
\end{lemma}
We are now in a position to show an improvement of flatness for solutions to equation \eqref{rescpb} by using the previous lemmas together with known regularity results for elliptic PDEs. Intuitively, we show that graphs of the solutions get more flat when we look at them in smaller balls.  

\begin{lemma}\label{flatle} There exist $\eps_0\in(0,1)$ and $\rho=\rho(p,n)\in(0,1)$ such that, for any viscosity solution $w$ of \eqref{rescpb}, $\osc_{B_1}(w)\leq 1$ and $||f||_{L^\infty(B_1)}\leq \eps_0$, there exists $q'\in\Rn$ such that
$$\underset{B_{\rho}}{\osc}\,(w(x)-q'\cdot x)\leq \frac{1}{2}\rho.$$
\end{lemma}

\begin{proof}
Thriving for a contradiction, assume that there exist a sequence of functions $(f_j)$ with $||f_j||_{L^\infty(B_1)} \rightarrow 0$, a sequence of vectors $(q_j)$ and a sequence of viscosity solutions $(w_j)$ with $\osc_{B_1}(w_j)\leq 1$ to
\begin{equation}
-\Delta w_j-(p-2) \left\langle D^2w_j\frac{D w_j+q_j}{\abs{D w_j+q_j}}, \frac{D w_j+q_j}{\abs{D w_j+q_j}}\right\rangle=f_j,
\end{equation}
such that, for all $q'\in \Rn$ and any $\rho\in(0,1)$
\begin{equation}\label{contra}
\underset{B_{\rho}}{\osc}(w_j(x)-q'\cdot x)>\frac{\rho}{2}.
\end{equation}
Using the compactness result of  Lemma \ref{compactres}, there exists a continuous function $w_{\infty}$ such that $w_j\to w_{\infty}$ uniformly in $B_\rho$ for any $\rho \in (0,1)$.
Passing to the limit in \eqref{contra}, we have that for any vector $q'$,
\begin{equation}\label{contrafin}
\underset{B_{\rho}}{\osc} (w_{\infty}(x)-q'\cdot x)>\dfrac{\rho}{2}.
\end{equation}

Suppose first that the sequence $(q_j)$ is bounded. Using  the result of Appendix \ref{appendixa1}, we extract a subsequence $(w_j)$ converging to a limit $w_{\infty}$, which satisfies
\[-\tr\left(\left(I+(p-2)\dfrac{Dw_\infty+q_\infty}{|Dw_\infty+q_\infty|}\otimes\dfrac{Dw_\infty+q_\infty}{|Dw_\infty+q_\infty|}\right )D^2 w_\infty\right) =0   \quad\text{in}\, B_1\]
in a viscosity sense. (Here $q_j\rightarrow q_\infty$ up to the same subsequence.) 
By the regularity result of Lemma \ref{lip2}, there exist $\beta_1=\beta_1(p,n)>0$ and $C_0=C_0(p,n)>0$ such that $\norm{w_\infty}_{C^{1,\beta_1}(B_{1/2})}\leq C_0.$

If the sequence $(q_j)$ is unbounded, we extract a converging subsequence from $e_j=\frac{q_j}{|q_j|}$, $e_j\rightarrow e_\infty$, and obtain (see Appendix \ref{appendixa2})
\begin{equation}\label{der2}
-\Delta w_{\infty}-(p-2) \left\langle D^2w_{\infty}\,e_{\infty}, e_{\infty}\right\rangle=0\qquad\text{in}\quad B_1,
\end{equation}
with $|e_{\infty}|=1$. Noticing that equation \eqref{der2} can be written as
\[
-\tr{((I+(p-2) e_{\infty}\otimes e_{\infty}) D^2w_\infty)}=0,
\]
we see that equation \eqref{der2} is uniformly elliptic and depends only on $D^2 w_\infty$. By the regularity result of \cite[Corollary 5.7]{caffarellicabrebook}, there is $\beta_2=\beta_2(p,n)>0$ so that $w_\infty\in C^{1,\beta_2}_\text{\text{loc}}$ and there exists $C_0=C_0(p,n)>0$ such that $\norm{w_\infty}_{C^{1,\beta_1}(B_{1/2})}\leq C_0$. 

We have shown that $w_\infty\in C^{1,\beta}_\text{\text{loc}}$ for $\beta=\min(\beta_1,\beta_2)>0$. Choose $\rho\in (0, 1/2)$ such that
\begin{equation}
C_0\rho^{\beta}\leq \frac{1}{4}.
\end{equation}
By $C^{1,\beta}_\text{\text{loc}}$ regularity, there exists a vector $k_{\rho}$ such that 
\begin{equation}
\underset{B_{\rho}}{\osc} (w_{\infty}(x)-k_{\rho}\cdot x)\leq C_0\rho^{1+\beta}\leq \frac{1}{4}\rho.
\end{equation}
This contradicts \eqref{contrafin} so the proof is complete.
\end{proof}
Proceeding by iteration, we obtain the following lemma. 

\begin{lemma}\label{lemiter}
Let $\rho$ and $\eps_0\in(0,1)$ be as in Lemma \ref{flatle} and let $u$ be a viscosity solution of \eqref{normpl} with $\osc_{B_1} (u)\leq1$ and $||f||_{L^\infty(B_1)}\leq \eps_0$. Then, there exists $\a\in(0,1)$ such that for all $k\in\N$, there exists $q_k\in \R^n$ such that
\begin{equation}\label{itera}
\underset{B_{r_{k}}}{\osc} \, (u(y)-q_{k}\cdot y)\leq  r_k^{1+\a },
\end{equation}
where $r_k:=\rho^k$.
\end{lemma}

\begin{proof}
For $k=0$, the estimate \eqref{itera} follows from the assumption $\osc_{B_1}(u) \leq 1$. Next we take $\alpha\in (0,1)$ such that $\rho^{\alpha}>1/2$.
We assume for $k\geq 0$ that we already constructed $q_k\in \R^n$ such that \eqref{itera} holds true. 
To prove the inductive step $k\rightarrow k+1$, we rescale the solution considering for $x\in B_1$
$$w_k(x)=r_k^{-1-\alpha}\big(u(r_k x)-q_k\cdot (r_k x)\big).$$
By induction assumption, we have $\underset{B_1}{\osc}\,(w_k)\leq 1$, and $w_k$ satisfies 
$$-\Delta w_k-(p-2) \left\langle D^2w_k\frac{D w_k+(q_k/r_k^{\alpha})}{\abs{D w_k+(q_k/r_k^{\alpha})}}, \frac{D w_k+(q_k/r_k^{\alpha})}{\abs{D w_k+(q_k/r_k^{\alpha})}}\right\rangle=f_k,  $$
where  $f_k(x)=r_k^{1-\alpha}f(r_k x)$ with $\norm{f_k}_{L^{\infty}(B_1)}\leq \eps_0$ since $\alpha<1$.
Using the result of Lemma \ref{flatle}, there exists $l_{k+1}\in\R^n$ such that
$$\underset{B_{\rho}}{\osc}\,(w_k(x)-l_{k+1}\cdot x)\leq \frac{1}{2}\rho.$$   
Setting $q_{k+1}=q_k+ l_{k+1} r_k^{\a}$, we get
\[
\underset{B_{r_{k+1}}}{\osc} \, (u(x)-q_{k+1}\cdot x)\leq \dfrac{\rho}{2} r_k^{1+\a }\leq r_{k+1}^{1+\a}.\qedhere
\]
\end{proof}

Since the estimate \eqref{itera} holds for every $k$, the proof of Theorem \ref{thm:main1} is complete.\\

The rest of the section is devoted to the proof of Lemma \ref{lip2}. First we need the following technical lemma concerning Lipschitz regularity of solutions of equation \eqref{rescpb} in the case $f\equiv 0$. For $n\times n$ matrices we use the matrix norm
\[
||A||:=\sup_{|x|\leq 1}\{|Ax|\}.
\] 

\begin{lemma}\label{lip1}
 	Assume that $f\equiv0$ and let $w$ be a viscosity solution to equation \eqref{rescpb} with $\underset{B_1}{\osc{w}}\leq 1$. For all $r\in (0,\frac 34)$, there exists a constant $ Q=Q(p,n)>0$ such that, if $|q|>Q$, then for all  $x,y\in B_r$,
 	\begin{equation}
 	\begin{split}
 	\abs{w(x)-w(y)}\le \tilde C\abs{x-y},
 	\end{split}
 	\end{equation}
 	where  $\tilde{C}=\tilde C(p,n)>0$.
 \end{lemma}

\begin{proof}
	We use the viscosity method introduced by Ishii and Lions in \cite{ishiilions}.
	\subsection*{Step 1}		 
It suffices to show that $w$ is Lipschitz in $B_{3/4}$, because this will imply that $w$ is Lipschitz in any smaller ball $B_\rho$ for $\rho\in \left(0,\frac34\right)$  with the same Lipschitz constant.
Take  $r=\frac 45$. First we fix $x_0, y_0\in B_{\frac{15r}{16}}$, where now $\frac{15 r}{16}=\frac34$, and introduce the auxiliary function 
	$$\Phi(x, y):=w(x)-w(y)-L\phi(\abs{x-y})-\frac M2\abs{x-x_0}^2-\frac M2\abs{y-y_0}^2,$$
	where $\phi$ is defined below. Our aim is to show that $\Phi(x, y)\leq 0$ for $(x,y)\in B_r\times B_r$.
	For a proper choice of $\phi$, this yields the desired regularity result.
	We take
	\[
	\begin{split}
	\phi(t)=
	\begin{cases}
	t-t^{\gamma}\phi_0& 0\le t\le t_1:=(\frac 1 {\gamma\phi_0})^{1/(\gamma-1)}  \\
	\phi(t_1)& \text{otherwise},
	\end{cases}
	\end{split}
	\]
	where $2>\gamma>1$ and $\phi_0>0$ is such that  $t_1\geq 2 $ and $\gamma \phi_02^{\gamma-1}\leq 1/4$. 
	
	Then
	\[
	\begin{split}
	\phi'(t)&=\begin{cases}
	1-\gamma t^{\gamma-1}\phi_0 & 0\le t\le t_1 \\
	0& \text{otherwise},
	\end{cases}\\
	\phi''(t)&=\begin{cases}
	-\gamma(\gamma-1)t^{\gamma-2} \phi_0& 0< t\le t_1 \\
	0 & \text{otherwise}.
	\end{cases}
	\end{split}
	\]
	In particular, $\phi'(t)\in  [\frac34,1]$ and $\phi''(t)<0$ when $t\in [0,2]$.
	
	\subsection*{Step 2} We argue by contradiction and assume that
	$\Phi$ has a positive
	maximum at some point $(x_1, y_1)\in \bar B_r\times \bar B_r$.
	Since $w$ is  continuous and
	its oscillation is bounded by 1, we get
	\begin{equation} 
	\begin{split}
	M\abs{x_1-x_0}^2\leq 2\osc_{B_1}{w}\leq 2,\\
	M\abs{y_1-y_0}^2\leq 2\osc_{B_1}{w}\leq 2.
	\end{split}
	\end{equation}
	Notice  that $x_1\neq y_1$, otherwise the maximum of $\Phi$ would be non positive.
	Choosing $M\geq\left(\dfrac{32}{r}\right)^2$, we have that $|x_1-x_0|< r/16$ and $|y_1-y_0|< r/16$ so that $x_1$ and $y_1$ are in $B_r$.
	
	We know that  $w$ is locally H\"older continuous and that  there exists a constant $C_{\beta}>0$ depending only on $p, n, r$ such that 
	$$|w(x)-w(y)|\leq C_\beta|x-y|^\beta \quad\text{for}\, x, y\in B_r.$$
	
	Using that $w$ is H\"older continuous, it follows, adjusting the constants (by  choosing $2M\leq C_{\beta}$), that  
	\begin{equation}\label{kulio}
	\begin{split}
	M\abs{x_1-x_0}\leq C_{\beta}&\abs{x_1-y_1}^{\beta/2},\\
	M\abs{y_1-y_0}\leq C_{\beta}&\abs{x_1-y_1}^{\beta/2}.
	\end{split}
	\end{equation}

		By  Jensen-Ishii's lemma (also known as theorem of sums, see \cite[Theorem 3.2]{crandall1992user}), there exist
		\[
		\begin{split}
		&(\tilde \zeta_x,X)\in \overline{\mathcal{J}}^{2,+}\left(w(x_1)-\frac M2\abs{x_1-x_0}^2\right),\\
		&( \tilde \zeta_y,Y)\in \overline{ \mathcal{J}}^{2,-}\left(w(y_1)+\frac M2\abs{y_1-y_0}^2\right),
		\end{split}
		\]
		that is 
		\[
		\begin{split}
		&(a,X+MI)\in \overline{\mathcal{J}}^{2,+}w(x_1),\\ &(b,Y-MI)\in \overline{\mathcal{J}}^{2,-}w(y_1),
		\end{split}
		\]
		where ($\tilde \zeta_x=\tilde \zeta_y$)
		\[
		\begin{split}
		a&=L\phi'(|x_1-y_1|) \frac{x_1-y_1}{\abs{x_1-y_1}}+M(x_1-x_0)=\tilde \zeta_x+M(x_1-x_0),\\
		b&=L\phi'(|x_1-y_1|) \frac{x_1-y_1}{\abs{x_1-y_1}}-M(y_1-y_0)=\tilde \zeta_y-M(y_1-y_0).
		\end{split}
		\]
		If $L$ is large enough (depending on the H\"older constant  $C_\beta$), we have
		\[
		\abs{a},\abs{b}\geq L\phi'(|x_1-y_1|) - C_\beta\abs{x_1-y_1}^{\beta/2}\ge \frac L2.
		\]
\noindent Moreover, by  Jensen-Ishii's lemma, for any $\tau>0$, we can take $X, Y\in \mathcal{S}^n$ such that 
		\begin{equation}\label{matriceineq1}
		- \big[\tau+2\norm{B}\big] \begin{pmatrix}
		I&0\\
		0&I 
		\end{pmatrix}\leq
		\begin{pmatrix}
		X&0\\
		0&-Y 
		\end{pmatrix}
		\end{equation}
		and
		\begin{align}\label{matineq2}
			\begin{pmatrix}
		X&0\\
		0&-Y 
		\end{pmatrix}
		\le 
		\begin{pmatrix}
		B&-B\\
		-B&B 
		\end{pmatrix}
		+\frac2\tau \begin{pmatrix}
		B^2&-B^2\\
		-B^2&B^2 
		\end{pmatrix}\\
		=D^2\phi (x_1, y_1)+\dfrac{1}{\tau} \left(D^2\phi(x_1, y_1)\right)^2,\nonumber
		\end{align}
		where 
		\begin{align*}	
		B=&L\phi''(|x_1-y_1|) \frac{x_1-y_1}{\abs{x_1-y_1}}\otimes \frac{x_1-y_1}{\abs{x_1-y_1}}\\
		&\quad +\frac{L\phi'(|x_1-y_1|)}{\abs{x_1-y_1}}\Bigg( I- \frac{x_1-y_1}{\abs{x_1-y_1}}\otimes \frac{x_1-y_1}{\abs{x_1-y_1}}\Bigg)
		\end{align*}	
		and 
		\begin{align*}	
		B^2=
		&\frac{L^2(\phi'(|x_1-y_1|))^2}{\abs{x_1-y_1}^2}\Bigg( I- \frac{x_1-y_1}{\abs{x_1-y_1}}\otimes \frac{x_1-y_1}{\abs{x_1-y_1}}\Bigg)\\
		&\quad +L^2(\phi''(|x_1-y_1|))^2 \frac{x_1-y_1}{\abs{x_1-y_1}}\otimes \frac{x_1-y_1}{\abs{x_1-y_1}}.
\end{align*}	
		Notice that  $\phi''(t)+\dfrac{\phi'(t)}{t}\geq 0$, $\phi''(t)\leq 0$ for $t\in (0,2)$ and hence
		\begin{equation}\label{lilou}
		\norm{B}\leq L \phi'(|x_1-y_1|),
		\end{equation}
		\begin{equation}\label{filou}
	 \norm{B^2}\leq L^2\left(|\phi''(|x_1-y_1|)|+\dfrac{\phi'(|x_1-y_1|)}{|x_1-y_1|}\right)^2.
		\end{equation}
	Moreover, for $\xi=\frac{x_1-y_1}{\abs{x_1-y_1}}$, we have
	
				\begin{equation*}
		\langle B\xi,\xi\rangle=L\phi''(|x_1-y_1|)<0, \qquad\langle B^2\xi,\xi\rangle=L^2(\phi''(|x_1-y_1|))^2.
		\end{equation*}
		Choosing $\tau=4L\left(|\phi''(|x_1-y_1|)|+\dfrac{\phi'(|x_1-y_1|)}{|x_1-y_1|}\right)$,  we have that for $\xi=\frac{x_1-y_1}{\abs{x_1-y_1}}$,
			\begin{align}\label{mercit}
		\langle B\xi,\xi\rangle +\frac2\tau \langle B^2\xi,\xi\rangle&=L\left(\phi''(|x_1-y_1|)+\frac2\tau L(\phi''(|x_1-y_1|))^2\right)\nonumber\\
		&\leq \dfrac{L}{2}\phi''(|x_1-y_1|)<0 .
		\end{align}
		In particular applying  inequalities \eqref{matriceineq1} and \eqref{matineq2} to any  vector $(\xi,\xi)$ with $\abs{\xi}=1$, we  have that $X- Y\leq 0$ and $\norm{X},\norm{Y}\leq 2\norm{B}+\tau$.  We refer the reader to \cite{ishiilions,crandall1992user} for details.\\
		Thus, setting $\eta_1=a+q$, $\eta_2=b+q$, we have for $\abs{q}$ large enough (depending only on $L$)
		\begin{align}\label{koivu}
		\abs{\eta_1}&\geq \abs{q}-\abs{a}\geq \frac{\abs{a}}{2}\geq \frac L4,\nonumber\\
		\abs{\eta_2}&\geq \abs{q}-\abs{b}\geq \frac{\abs{b}}{2}\geq \frac L4,
		\end{align}
		where $L$ will be chosen later on and $L$ will depend only on $p,n, C_\beta$.
		We write the  viscosity inequalities
		\begin{equation*}
		\begin{split}
		0&\leq  \tr (X+MI)+(p-2)\dfrac{\left\langle(X+MI) (a+q), (a+q)\right\rangle}{|a+q|^2}\\
		0&\geq  \tr(Y-MI)+(p-2) \dfrac{\left\langle(Y-MI) (b+q), (b+q)\right\rangle}{|b+q|^2}.
		\end{split}
		\end{equation*}
	 In other words 
		\begin{equation*}
		\begin{split}
		0&\leq  \tr (A(\eta_1)(X+MI))\\
		0&\leq  -\tr (A(\eta_2)(Y-MI))
		\end{split}
		\end{equation*}
		where  for $\eta \neq 0$ $\bar \eta=\dfrac{\eta}{|\eta |}$ and 
		\[A(\eta):= I+(p-2)\ol\eta\otimes \ol\eta.\]
		Adding the two inequalities, we get 
		\begin{equation*}
		0 \leq  \tr (A(\eta_1)(X+MI))
		-\tr (A(\eta_2)(Y-MI)).
		\end{equation*}
		It follows that 
		\begin{align}\label{gregory}
		0 \leq  &\underbrace{\tr (A(\eta_1)(X-Y))}_{(1)}
		+\underbrace{tr ((A(\eta_1)-A(\eta_2))Y)}_{(2)}\nonumber\\
		&+\underbrace{M\big[\tr (A(\eta_1))+\tr (A(\eta_2))}_{(3)} \big].
		\end{align}
		
\noindent\textbf{Estimate of (1)}.		Notice that all the eigenvalues of $X-Y$ are non positive. Moreover,  applying the previous matrix inequality \eqref{matineq2} to the vector $(\xi,-\xi)$ where $\xi:=\frac{x_1-y_1}{|x_1-y_1|}$  and using \eqref{mercit}, 
		we obtain
		\begin{align}\label{camille}
		\langle (X-Y) \xi, \xi\rangle&\leq 4\left(\langle B\xi,\xi\rangle+\frac2\tau\langle B^2\xi,\xi\rangle)\right)\nonumber \\
		&\leq 2 L\phi''(|x_1-y_1|)<0.
		\end{align}
		Hence  at least one of the eigenvalue of $X-Y$  that we denote by  $\lambda_{i_0}$ is   negative and smaller than $2 L\phi''(|x_1-y_1|)$. The eigenvalues of $A(\eta_1)$ belong to $[\min(1, p-1), \max(1, p-1)]$.	Using \eqref{camille}, it follows by \cite{theo} that 
		\begin{align*}  
		\tr(A(\eta_1) (X-Y))&\leq \sum_i \lambda_i(A(\eta_1))\lambda_i(X-Y)\\
		&\leq \min(1, p-1)\lambda_{i_0}(X-Y)\\
		&\leq 2\min(1, p-1) L \phi''(|x_1-y_1|).
		\end{align*}
		
\noindent\textbf{Estimate of (2)}.		We have
		\begin{align*}A(\eta_1)-A(\eta_2)&=(\ol\eta_1\otimes \ol\eta_1-\ol\eta_2\otimes \ol\eta_2)(p-2)\\
	&=[(\ol\eta_1-\ol\eta_2+\ol\eta_2)\otimes\ol\eta_1-\ol\eta_2\otimes(\ol\eta_2-\ol\eta_1+
	\ol\eta_1)](p-2)\\
	&=[(\ol\eta_1-\ol\eta_2)\otimes\ol\eta_1+\ol\eta_2\otimes\ol\eta_1
	-\ol\eta_2\otimes(\ol\eta_2-\ol\eta_1)-\ol\eta_2\otimes\ol\eta_1](p-2)\\
	&=[(\ol\eta_1-\ol\eta_2)\otimes\ol\eta_1
	-\ol\eta_2\otimes(\ol\eta_2-\ol\eta_1)](p-2).
		\end{align*}
		Hence,
		\begin{align*}
		\tr( (A(\eta_1)-A(\eta_2)) Y)&\leq n\norm{Y}
		\norm{A(\eta_1)-A(\eta_2)}  \\
		&\leq n\abs{p-2}\norm{Y}|\ol\eta_1-\ol\eta_2|\left( |\ol\eta_1|+|\ol\eta_2|\right)\\
		&\leq 2n\abs{p-2}\norm{Y}|\ol \eta_1-\ol\eta_2|.
		\end{align*}
		On one hand we have
		\begin{equation*}
		\begin{split}
		\abs{\ol \eta_1-\ol \eta_2}&=
		\abs{\frac{\eta_1}{\abs {\eta_1}}-\frac{\eta_2}{\abs {\eta_2}}}
\le \max\left( \frac{\abs{\eta_2- \eta_1}}{\abs{\eta_2}},\frac{ \abs{\eta_2- \eta_1}}{\abs{\eta_1}}\right)\\
		&\le \frac {8C_\beta}{ L}\abs{x_1-y_1}^{\beta/2},
		\end{split}
		\end{equation*}
		where we used \eqref{koivu} and \eqref{kulio}.\\
		
		On the other hand, by \eqref{matriceineq1}--\eqref{filou},
		\begin{align*}
		\norm{Y}&=\max_{\ol \xi} |\langle Y\ol \xi, \ol \xi\rangle|
		\le 2 |\langle B\ol \xi,\ol \xi \rangle|+\frac4\tau|\langle B^2\ol \xi,\ol \xi \rangle| \\
		&\leq 4L\left( \frac{\phi'(|x_1-y_1|)}{\abs{x_1-y_1}}+ |\phi''(|x_1-y_1|)|\right).
		\end{align*}
		Hence, remembering that $|x_1-y_1|\leq 2$, we end up with
	\begin{align*}
	 \tr( (A(\eta_1)-A(\eta_2)) Y)&\leq 128n\abs{p-2}C_\beta \phi'(|x_1-y_1|) \abs{x_1-y_1}^{-1+\beta/2}\\
	 &\quad+128n\abs{p-2}C_\beta |\phi''(|x_1-y_1|)|.
	 \end{align*}
		 \textbf{Estimate of (3)}. Finally, we have
		$$ M(\tr(A(\eta_1))+\tr(A(\eta_2)))\leq 2Mn\max(1, p-1).$$

		\subsection*{Step 3}
		
		Gathering the previous estimates with \eqref{gregory} and recalling the definition of $\phi$, we get 
		\begin{align*}
		0&\leq 128n\abs{p-2}C_\beta\left(\phi'(|x_1-y_1|)  \abs{x_1-y_1}^{\beta/2-1}+  |\phi''(|x_1-y_1|)|\right)\\
		&\quad+2\min(1, p-1) L \phi''(|x_1-y_1|)  + +2Mn\max(1, p-1)\\
		&\leq 128n\abs{p-2}C_\beta  \abs{x_1-y_1}^{\beta/2-1}+2nM\max(1, p-1)\\
		&\quad +128n\abs{p-2}C_\beta\gamma(\gamma-1)\phi_0\abs{x_1-y_1}^{\gamma-2}\\
		&\quad-2\min(1, p-1)  \gamma(\gamma-1)\phi_0L\abs{x_1-y_1}^{\gamma-2}.
		\end{align*}
		Taking $\gamma=1+\beta/2>1$ and choosing $L$ large enough depending on $p,n, C_\beta$, we get
		$$ 0\leq \dfrac{-\min(1, p-1)\gamma(\gamma-1)\phi_0}{200} L\abs{x_1-y_1}^{\gamma-2}<0,   $$
		which is  a contradiction. Hence, by choosing first $L$ such that 
		\begin{align*}
		0&>128n\abs{p-2}C_\beta \left(\phi'(|x_1-y_1|)  \abs{x_1-y_1}^{\beta/2-1}+ |\phi''(|x_1-y_1|)|\right)\\
		&\quad +\min(1, p-1) L \phi''(|x_1-y_1|)+2nM\max(1, p-1)
		\end{align*}
		and then taking $|q|$ large enough (depending on $L$, it suffices that $|q|> 6L >\frac32 |a|$ see \eqref{koivu}), we reach a contradiction and   hence $\Phi(x,y)\leq 0$ for $(x,y)\in B_r\times B_{r}$. The desired result follows since for $x_0,y_0\in B_{\frac{15r}{16}}$, we have $\Phi(x_0,y_0)\leq 0$, we get
		\[
		|w(x_0)-w(y_0)|\leq L\phi(|x_0-y_0|)\leq L|x_0-y_0|.
		\]
		Remembering that $\frac{15r}{16}=\frac{15\cdot 4}{16\cdot 5}=\frac34$, we get that $w$ is Lipschitz in $B_{\frac34}$.
\end{proof}

Finally, once we have a control on the Lipschitz norm of $w$, we can prove Lemma \ref{lip2}.\\

\noindent \emph{Proof of Lemma \ref{lip2}.}
Introducing the function $v(x):=w(x)+q\cdot x$, we notice that $v$ is  a viscosity solution to 
	\[-\Delta_p^N v=0\qquad\text{in}\, B_1,\]
	and thus also a viscosity solution to the homogeneous $p$-Laplace equation $\Delta_p v=0$, see \cite{juutinenlm01}. By the equivalence result first proved by \cite{juutinenlm01}, $v$ is a weak solution to the homogeneous $p$-Laplace equation. By the classical regularity result, there is $\beta_1=\beta_1(p,n)>0$ so that $v\in C^{1,\beta_1}_{\text{\text{loc}}}(B_1)$ and hence  also $w\in C^{1,\beta_1}_{\text{loc}}(B_1)$. 
	The main difficulty is to provide  $C^{1, \beta_1}$ estimates which are uniform with respect to $q$. 
	
	We notice that if $|q|$ is  large enough, then the equation satisfied by $w$ is uniformly elliptic and the operator is not discontinuous. Taking $Q$ from Lemma \ref{lip1} and assuming that $|q|>Q$, we know from Lemma \ref{lip1} that $|Dw(x)|$ is controlled by some constant $\tilde C$ depending only on $p,n$ and independent of $|q|$ for any $x\in B_{3/4}$. 	It follows that,   if $q$ satisfies
	\[|q|\geq\theta_0:=\max(Q, 2\tilde C)\geq 2\norm{Dw}_{L^{\infty}(B_{3/4})},\] then
	denoting $\nu:=\frac{1}{\abs{q}}$ and $e:=\frac{q}{\abs{q}}$, we have
	$$\frac 12 \leq |e|-|\nu Dw|\leq|\nu Dw+e|\leq\abs{e} +\abs{\nu Dw}\leq \frac32.$$
	Defining $$\Sigma (x):= (p-2) \dfrac{Dw(x)+q}{|Dw(x)+q|}\otimes\dfrac{Dw(x)+q}{|Dw(x)+q|},$$
	we note that \eqref{rescpb} can be rewritten  as
	$$-\tr(F(D^2 w,x) )=0,$$
	where $F:\mathcal{S}^n\times B_{3/4}\rightarrow \R$,
	$$F(M,x)=-\tr((I +\Sigma(x))M),$$
	is continuous.
	
	Since $Dw$ is H\"older continuous, 
	we can see this equation as a linear elliptic equation
	with  $C^{\alpha}$
	coefficients. The standard Calder\'on-Zygmund theory
	provides local $C^{2,\alpha}$
	regularity on $w$ (bootstrapping the argument gives even  $C^{\infty}$
	regularity on $w$).
	
	Moreover, since $v$ is a weak solution to the usual $p$-Laplacian, it follows that $w$ is a weak solution to 
	\begin{equation}\label{rftghy}
	-\operatorname{div} \left(| Dw +q|^{p-2} (Dw+q)\right)=0\quad \text{in}\quad B_{3/4}.
	\end{equation}
	Writing the weak formulation, we have that for any test function $\varphi\in C^{\infty}_0(B_{3/4})$
	\begin{equation}
	\int_{B_{3/4}}|Dw+q|^{p-2}(Dw+q)\cdot D\varphi \, dx=0.
	\end{equation}
	
	Fixing $k$, $1\leq k\leq n$, taking $\varphi_k=
	\dfrac{\partial \varphi}{\partial x_k}$ instead of $\varphi$ as a test function and integrating by parts, we obtain 
	$$\int_{B_{3/4}}(|Dw+q|^{p-2}(I+\Sigma(x)) \, Dw_k)\cdot D\varphi\, dx=0.$$
	Dividing  by $|q|^{p-2}$, we have for any function $\varphi\in C^\infty_0 (B_{3/4})$
	$$\int_{B_{3/4}}\left(|\nu Dw+e|^{p-2}(I+\Sigma(x))Dw_k\right)\cdot D\varphi\, dx=0.$$
	We conclude that $h:= w_k$ is a weak solution to the linear uniformly elliptic equation
	$$-\operatorname{div}( \mathbf{A}(x)\, Dh)=0,$$
	where $\mathbf{A}(x):=|\nu Dw(x)+e|^{p-2}(I+\Sigma(x))\in \mathcal{S}^n$ satisfies
	\begin{align*}
	\mathbf{A}(x)&\geq\min(1, p-1)\min\left(\left(\frac{ 3}{ 2}\right)^{p-2},\left(\frac 1 2\right)^{p-2}\right)I,\\
	\mathbf{A}(x)&\leq\max(1,p-1) \max\left(\left(\frac{ 3}{ 2}\right)^{p-2},\left(\frac 1 2\right)^{p-2}\right)I.
	\end{align*}
	Using the classical result of De Giorgi \cite{Degio57}  for uniformly elliptic equations with bounded coefficients  (see also \cite{moser60new}, \cite[Theorems 8.24, 12.1]{gilbargt01})
	we get that $h$ is locally H\"older continuous and 
	\begin{equation}
	[h]_{C^{0,\beta_1}(B_{1/2})}\leq C(p,n)\norm{h}_{L^2(B_{3/4})}
	\end{equation}
	for some $\beta_1=\beta_1(p,n)>0$.
	
	We conclude  that if $|q|> \theta_0=\theta_0(p,n)$ then there exist $\beta_1=\beta_1(n,p)>0$ and 
	$C=C(p,n, \norm{w}_{L^\infty(B_1)},\norm{D w}_{L^\infty(B_{3/4})})=  C_0(p,n)>0$ (see Lemma \ref{lip1})
	such that 
	$$[w]_{C^{1, \beta_1}(B_{1/2})}\leq C_0.$$
	
	If $|q|\leq \theta_0$, we have \[\underset{B_1}{\osc}\, v\leq \underset{B_1}{\osc w}+ 2|q|\leq 1+2\theta_0.\]
	It follows that 
	\[
	[w]_{C^{1, \beta_1}(B_{1/2})}\leq [v]_{C^{1, \beta_1}(B_{1/2})}+2|q|\leq C(p,n)\underset{B_1}{\osc v}+2\theta_0\leq C_0(p,n).\]\qed

\subsection{Second proof by using distributional weak theory}\label{section3.2}
In this part we establish a second method to show that viscosity solutions to \eqref{normpl} are in $C^{1, \a}_{\text{\text{loc}}}(\Omega)$, when $f\in L^{\infty}(\Om)\cap C(\Om)$ and $p\geq 2$. Recall that equation \eqref{usualpl} reads as
\[
-\Delta_p u=|D u|^{p-2} f\qquad\text{in}\quad\Omega.
\]
Since the exponent of the nonlinear gradient term is less than $p$ and $f\in L^\infty(\Omega)$, locally H\"older continuous weak solutions of \eqref{usualpl} are known to be of class $C^{1,\a}_{\text{loc}}$ for some $\alpha\in (0, 1)$, see \cite{Tolk84}. More precisely, if $u$ is a weak solution to \eqref{usualpl} in $B_{2r}$, then
\[
[u]_{C^{1,\a}(B_r)}\leq C=C\left(p,n,r, \norm{u}_{L^{\infty}(B_{2r})}, \norm{f}_{L^{\infty}(B_{2r})}\right).
\]

We know that in the case $p\geq 2$ viscosity solutions of \eqref{normpl} are viscosity solutions to \eqref{usualpl}, and our aim is to show that they are also weak solutions to \eqref{usualpl}. The next theorem holds for the more general case $f\in L^q(\Om)\cap C(\Om)$, where $q>\max(n,p/2)$, and will be useful not only in this subsection, but in Section \ref{chpt4} and Section \ref{chpt6} as well. Our proof cannot rely on uniqueness, see Example \ref{esim} below. Instead, we use a technique developed by Julin and Juutinen in \cite{julin2012new}. We point out that the uniqueness of viscosity solutions is known only  when $f$ is either 0 or has constant sign (see \cite{kamamik2012}). A detailed  discussion can be found in \cite{discut1, peresssw09} for the case of the  normalized infinity Laplacian.

\begin{theorem}\label{equivalence3.4} Assume that $p\geq 2$, $\max(n,p/2)<q\leq\infty$, and $f\in L^q(\Om)\cap C(\Om)$.
Let $u$ be a bounded viscosity solution to \eqref{normpl}. Then $u$ is a weak solution to \eqref{usualpl}.
\end{theorem}

\begin{proof}
We will prove that a viscosity supersolution $u$ to \eqref{usualpl} is also a weak supersolution to \eqref{usualpl} (the
proof adapts to the  case of subsolutions with obvious changes). We need to show that 
\begin{equation*}
\int_{\Omega} |Du|^{p-2} Du\cdot D\varphi\, dx\geq \int_{\Omega}|Du|^{p-2} f\varphi\, dx,
\end{equation*}
where $\varphi\in C^{\infty}_0(\Omega)$.

\subsection*{Step 1: regularization.} Let us start by showing that the inf-convolution $u_\eps$ of $u$,  
\begin{equation}
u_{\eps}(x):=\underset{y\in\Omega}{\inf}\left( u(y)+\dfrac{|x-y|^2}{2\eps}\right),
\end{equation}
is a weak supersolution to
\begin{equation}\label{infconveq}
-\Delta_p u_\eps\geq |D u_\eps|^{p-2}f_{\eps} \qquad\text{in}\quad \Omega_{r(\eps)},
\end{equation}
where $$f_{\eps}(x)=\underset{|y-x|\leq 2\sqrt{\eps \osc_\Omega{u}}}{\inf}\, f(y)$$ and 
$$\Om_{r(\eps)}=\left\{x\ :\  \dist(x, \partial\Om)> 2\sqrt{\eps \osc_\Omega{u}}\right\}.$$

We recall some properties of inf-convolutions. For more general discussion and proofs, see the appendix of \cite{julin2012new}. First we mention that $u_{\eps}$ is a semi-concave viscosity supersolution to \eqref{infconveq}. Moreover, $u_{\eps}\in W^{1, \infty}_{\text{loc}}(\Omega_{r(\eps)})$ is  twice differentiable a.e and satisfies
\begin{align}
-\Delta_p u_{\eps}&=-|Du_{\eps}|^{p-2}\left(\Delta u_{\eps}+(p-2) D^2u_{\eps}\dfrac{Du_{\eps}}{|Du_{\eps}|}\cdot \dfrac{Du_{\eps}}{|D u_{\eps}|}\right)\nonumber\\
& \geq |Du_{\eps}|^{p-2} f_{\eps}
\end{align}
a.e. in $\Omega_{r(\eps)}$. Finally we mention that $u_{\eps}\to u$ locally uniformly and $||u_\eps||_{L^\infty(\Om_{r(\eps)})}\leq ||u||_{L^\infty(\Om)}$, see \cite{crandall1992user}. 

Since the function $\phi(x):=u_{\eps}(x)-\dfrac{1}{2\eps}|x|^2$ is concave in $\Omega_{r(\eps)}$, we can approximate it by a sequence $(\phi_j)$ of smooth concave functions by using standard mollification. Denoting $u_{\eps, j}:=\phi_j+\dfrac{1}{2\eps}|x|^2$, we can integrate by parts to obtain 
\begin{equation}\label{hep1}
\int_{\Om_{r(\eps)}} |Du_{\eps, j}|^{p-2} Du_{\eps,j}\cdot D\varphi\, dx=\int_{\Om_{r(\eps)}}(-\Delta_p u_{\eps,j})\varphi\, dx.
\end{equation}
Since $Du_{\eps}$ is locally bounded, the dominated convergence theorem implies
\begin{equation}\label{hep2}
\underset{j\to\infty}{\lim} \int_{\Om_{r(\eps)}}|Du_{\eps,j}|^{p-2} Du_{\eps,j}\cdot D\varphi\, dx=\int_{\Om_{r(\eps)}} |Du_{\eps}|^{p-2} Du_{\eps}\cdot D\varphi\, dx.
\end{equation}
Next, using the concavity of $u_{\eps, j}$ (we have $D^2u_{\eps, j}\leq \frac{1}{\eps} I$) and the local boundedness of $D u_{\eps,j}$, we get $$ -\Delta_p u_{\eps, j}\geq -\dfrac{C^{p-2}(n+p-2)}{\eps}$$ locally in $\Om_{r(\eps)}$.
Applying Fatou's lemma, we obtain
\begin{equation}\label{hep3}
\underset{j\to\infty}{\liminf}\int_{\Om_{r(\eps)}} (-\Delta_p u_{\eps,j})\varphi\,dx\geq \int_{\Om_{r(\eps)}} \underset{j\to\infty}{\liminf} (-\Delta_p u_{\eps,j})\varphi\, dx.
\end{equation}
Since
$$\liminf_{j\to\infty}(-\Delta_p u_{\eps, j}(x))=-\Delta_p u_{\eps}(x)$$
almost everywhere, by using \eqref{hep1}, \eqref{hep2} and \eqref{hep3} we obtain
\begin{align*}
\int_{\Om_{r(\eps)}} |Du_{\eps}|^{p-2} Du_{\eps}\cdot D\varphi\, dx &\geq \int_{\Om_{r(\eps)}}(-\Delta_p u_{\eps})\varphi\, dx\nonumber\\
& \geq \int_{\Om_{r(\eps)}}|Du_{\eps}|^{p-2} f_{\eps}\varphi\, dx.
\end{align*}
Hence, we have shown that $u_\eps\in W^{1,p}_{\text{loc}}(\Om_{r_\eps})$ is a weak  supersolution to \eqref{infconveq}.

\subsection*{Step 2: passing to the limit in the regularization.} Take an arbitrary test function $\varphi\in C^{\infty}_0(\Omega)$. We finish the proof by showing that
\begin{equation}\label{passlimit1}
\int_{\Om_{r(\eps)}} |Du_\eps|^{p-2} Du_\eps \cdot D\varphi\, dx\rightarrow \int_{\Om} |Du|^{p-2} Du \cdot D\varphi\, dx
\end{equation}
and
\begin{equation}\label{passlimit2}
\int_{\Om_{r(\eps)}}|Du_\eps|^{p-2} f_\eps \varphi\, dx \rightarrow \int_{\Om}|Du|^{p-2} f \varphi\, dx.
\end{equation}

Let $\Om''$ be the support of $\varphi$ and $\eps$ so small that $\Om''\subset \Om'\subset \subset \Om_{r(\eps)}$. We start by showing that $D u_\eps$ is uniformly bounded in $L^p(\Om')$. Take a compactly supported smooth cut-off function $\xi:\Om_{r(\eps)}\rightarrow [0,1]$ such that $\xi\equiv 1$ on $\Om''$ and such that the support of $\xi$ is included in $\Om'$. Choose the test function $(2L-u_{\eps})\xi^p$ in the weak formulation, where $L=\sup_{\Om'}|u_\eps|$. By using H\"older's inequality we obtain
\begin{eqnarray*}
\int_{\Om_{r(\eps)}} \xi^p|D u_{\eps}|^p\,dx &\leq&\int_{\Om_{r(\eps)}} |D u_{\eps}|^{p-2}(2L-u_{\eps})\xi^2\xi^{p-2}|f_{\eps}|\, dx\\
&\quad& +\ p\int_{\Om_{r(\eps)}} \xi^{p-1}|Du_{\eps}|^{p-2} Du_{\eps}\cdot D\xi  (2L-u_{\eps})\, dx \\
&\leq&1/4\int_{\Om_{r(\eps)}} \xi^p|D u_{\eps}|^p\,dx+ C(p) L^{p/2} \int_{\Om_{r(\eps)}}|f_\eps|^{p/2} \xi^p\,dx\\
&\quad& +\ C(p)\int_{\Om_{r(\eps)}} L^p|D\xi|^{p} dx +1/4\int_{\Om_{r(\eps)}} \xi^p|D u_{\eps}|^p\,dx.
\end{eqnarray*}
It follows that 
\begin{align}\label{vaartyu}
\int_{\Om_{r(\eps)}} \xi^p|D u_{\eps}|^p\,dx &\leq C(p) L^{p/2} \int_{\Om_{r(\eps)}}|f_\eps|^{p/2} \xi^p\,dx
 +\ C(p)\int_{\Om_{r(\eps)}} L^p|D\xi|^{p} dx\nonumber\\
 &\leq C= C\left(p,n, \norm{u}_{L^\infty(\Om)}, \norm{f}_{L^q(\Om)}\right). 
\end{align}
Hence, $D u_{\eps}$ is uniformly bounded with respect to $\eps$ in $L^p(\Om')$. It follows that there exists a subsequence such that $D u_{\eps}\to D u$ weakly in $L^{p}(\Om')$, and we can also show that $D u_{\eps}\to D u$ strongly in $L^{p}(\Om')$. Indeed, taking this time the test function $(u-u_{\eps})\xi^p$, we estimate

\begin{align*}
-\int_{\Om_{r(\eps)}} \xi^p|D u_{\eps}|^{p-2}D u_{\eps}\cdot(Du-Du_{\eps})\,dx &\leq\int_{\Om_{r(\eps)}}
 |D u_{\eps}|^{p-2}(u-u_{\eps})\xi^{p}|f_{\eps}|\, dx\\
&+p\int_{\Om_{r(\eps)}} \xi^{p-1}|Du_{\eps}|^{p-1}
| D\xi | (u-u_{\eps})\, dx.
\end{align*}

\noindent Adding $\int_{\Omega_{r(\eps)}}|Du|^{p-2} Du\cdot (Du- Du_{\eps})\xi^p\, dx$ to this inequality and recalling that  for $p>2$
$$(|a|^{p-2}a- |b|^{p-2}b)\cdot (a-b)\geq C(p) |a-b|^p,$$
we get 
\begin{align}
C(p)\int_{\Om_{r(\eps)}} |Du -Du_{\eps}|^p\xi^p\, dx
&\leq \norm{u-u_{\eps}}_{L^\infty(\Om')}\norm{D u_{\eps}\xi}^{p-2}_{L^p(\Om')}\norm{f_\eps \xi}_{L^{p/2}(\Om')}\nonumber\\
&+\ p\norm{u-u_{\eps}}_{L^\infty(\Om')}\norm{D u_{\eps}\xi}^{p-1}_{L^p(\Om')} \norm{D \xi}_{L^p(\Om')}\label{convpass}\nonumber\\
&+\ \int_{\Om_{r(\eps)}}|Du|^{p-2} Du\cdot (Du- Du_{\eps})\xi^p\, dx.\nonumber
\end{align}
By using the local uniform convergence of $u_{\eps}$ to $u$, the facts $Du\in L^{p}(\Om')$, $\norm{f_{\eps}}_{L^q(\Om')}\leq C(q,\Omega)\norm{f}_{L^q(\Omega)}$ and the weak convergence of $Du_{\eps}$ in $L^p(\Om')$, we obtain
\[
\int_{\Om_{r(\eps)}} |Du -Du_{\eps}|^p\xi^p\, dx\rightarrow 0,
\]
so $Du_{\eps}\to Du$ strongly in $L^p(\Om')$. 

Finally, we are ready to show that \eqref{passlimit1} and \eqref{passlimit2} hold. First we use the triangle inequality to obtain
\begin{align*}
&\left|\int_{\Om'}|Du_\eps|^{p-2} f_\eps \varphi\, dx- \int_{\Om'}|Du|^{p-2} f \varphi\, dx \right|\\
& \leq \left|\int_{\Om'}|D u_\eps|^{p-2}(f_\eps-f)\varphi\, dx\right|+\left|\int_{\Om'}( |D u_\eps|^{p-2}-|D u|^{p-2}) f \varphi\, dx\right|\\
& =:I_1+I_2.
\end{align*}
Using the generalized H\"older's inequality, we get
\[
I_1\leq ||D u_\eps||_{L^p(\Om')}||f_\eps-f||_{L^q(\Om)}||\varphi||_{L^\infty(\Om')}\leq C||f_\eps-f||_{L^{p/2}(\Om')}\rightarrow 0.
\]
To estimate $I_2$, notice first that since $f$ and $\varphi$ are continuous in $\Om'$, $f\varphi$ is bounded in $\Om'$. By using H\"older's inequality and the convexity of $\frac{p}{p-2}$ power function, we obtain
\begin{align*}
I_2 &\leq C\norm{|Du_\eps|^{p-2}-|D u|^{p-2}}_{L^{\frac{p}{p-2}}(\Om')}\\
& \leq C\left|\norm{D u_\eps}^p_{L^p(\Om')}-\norm{D u}^p_{L^p(\Om')}\right|^{\frac{p-2}{p}}\rightarrow 0,
\end{align*}
since $D u_\eps\rightarrow D u$ in $L^p(\Om')$. Hence, \eqref{passlimit2} holds, and by using the same argument, also \eqref{passlimit1} holds. The proof is complete. \qedhere
 \end{proof}

Finally, we give an example to show why we deliberately avoided using the uniqueness machinery. For similar counterexamples in the case of the infinity Laplacian, see \cite{crastaf15}.

\begin{example}\label{esim}
We give an example to show that for given continuous boundary data, there can be several weak solutions to equation \eqref{usualpl}. Let $f= (p-1)$.
Consider the 1-dimensional situation, where for $R\in [0,1]$ we  define a function
\[
\begin{split}
u(x)=\begin{cases}
C-C(\frac{{x+R}}{-1+R})^2 & x\in (-1,\textcolor{blue}{-}R)\\
 C& [-R,R]\\
 C-C(\frac{{x-R}}{1-R})^2&x\in (R,1).
\end{cases}
\end{split}
\]
Solving $C$ from
\[
\begin{split}
-(p-1)u''=(p-1)
\end{split}
\]
gives 
\[
\begin{split}
\frac{2C}{(-1+R)^2}=1 \text{ i.e. }C=\half(-1+R)^2.
\end{split}
\]
This gives different weak solutions  for the whole range of $R$. Indeed, assuming that $u\in W^{1,p}((-1,1))$, for any test function $\varphi\in C^{\infty}_0\left((-1,1)\right)$
\[
\begin{split}
\int_{-1}^{1} |u'|^{p-2}u'\varphi'\, dx&=-\int_{-1}^{-R} (x+R)^{p-1}\varphi'(x)\, dx+\int_R^1 (x-R)^{p-1}\varphi'(x)\,dx\\
&=(p-1)\Big(\int_{-1}^{-R} (x+R)^{p-2}\varphi(x)\, dx\\
&\quad +\int_R^1 (R-x)^{p-2}\varphi(x)\,dx\Big)\\
&=\int_{-1}^1 |u'|^{p-2}\varphi f\ dx.
\end{split}
\]
Only the largest i.e. $R=0$ is a solution to the original $-\Delta_p^N u=(p-1)$.

\noindent This counterexample also shows that in general weak solutions to \eqref{usualpl} are not necessary viscosity solutions to \eqref{normpl}.
\end{example}

\section{Uniform gradient estimates when $f\in C(\Om)\cap  L ^q(\Om)$}\label{chpt4}

In this section we assume that $p>2$, $f\in C(\Om)\cap L^q(\Om)$ for some $q>\max\left(n, \frac{p}{2}, 2\right)$. Our aim is to prove Theorem \ref{thm:main2}, which states that viscosity solutions of \eqref{normpl} are of class $C^{1,\alpha}_{\text{loc}}(\Omega)$ for some $\alpha =\alpha(p,q,n)$, and for any $\Omega''\subset \Om' \subset \subset \Omega$,
\[
[u]_{C^{1,\beta}(\Omega'')}\leq C=C\left(p,q,n,d,d'', ||u||_{L^\infty(\Om)},\norm{f}_{L^q(\Om)}\right),
\]  
where $d=\text{diam}\,(\Om)$ and $d''=\text{dist}\,(\Om'',\partial \Om')$.

Let\ $u$ be a viscosity solution of \eqref{normpl}.  From Lemma \ref{compactres}, we know that $u$ is locally of class $C^{0,\beta}$ for some $\beta=\beta(p,n)$.
From Section \ref{chpt3}, we know that $u$ is a weak solution to \eqref{usualpl} and passing to the limit in \eqref{vaartyu}, we know that for any $\Om'\subset\subset \Om$,
\begin{equation}\label{metrorome}
\norm{Du}_{L^p(\Om')}\leq C(p,n, \norm{u}_{L^\infty(\Om)}, \norm{f}_{L^q(\Om)}).
\end{equation}
 Moreover, for any $\lambda>0$ the function $u$ is a bounded viscosity solution to the following equation
\begin{equation}\label{eigennormpl}
-\Delta_p^N v(x)+\lambda v(x)=h(x):=f(x)+\lambda u(x),  \quad x\in \Omega.
\end{equation}
 Let $\Om'\subset\subset \Om$ with $\Om'$ smooth enough so that weak solutions to \eqref{usualpl} satisfy the boundary condition in a classical sense. In the sequel we fix small enough $\lambda=\lambda (p,n, \Omega')>0$ and  a viscosity solution $u$ of \eqref{normpl}. We take H\"older continuous functions $f_\eps\in C(\Om)\cap L^q(\Om)$ such that $f_\eps$ converges uniformly to $f$ in $\Om'$  and $f_\eps$ converges to $f$ in $L^q(\Om')$. 
The idea for the proof of Theorem \ref{thm:main2} is to obtain uniform estimates for solutions $v_\eps$ to the following regularized problems
\begin{numcases}{}
 -\mathrm{div} \left(\left( | D v_{\eps}|^2+ \eps^2\right)^{(p-2)/2}  D v_{\eps} \right)=(|D v_{\eps}|^{2}+\eps^2)^{\frac{p-2}{2}}(h_\eps-\lambda v_\eps),  \quad x\in \Omega',\nonumber\\
v_{\eps}=u \qquad\qquad\qquad\hspace{40pt} x\in\partial\Om',\label{regutpb}
\end{numcases}
where  $h_\eps=f_\eps+\lambda u$. Notice that the right-hand side of the equation has a growth of power less than $p$ with respect to the gradient, and $h_\eps $ is  bounded. Since the regularized equations are uniformly elliptic with smooth coefficients, in Step 1 we notice that $v_{\eps}\in C^{1,\beta(\eps)}_{\text{loc}}(\Om')\cap W^{2,2}_{\text{loc}}(\Om')$. In the next two steps we obtain uniform estimate for the norm $\norm{D v_{\eps}}_{L^p(\Omega')}$ and local Lipschitz estimate for $v_\eps$. Once we know that $v_{\eps}$ and  $|D v_{\eps}|^{p-2}$ are locally uniformly bounded, in Step 4 we use the regularity result of Lieberman \cite{Lie93} to get a local uniform H\"older estimate for the gradient $D v_\eps$. By using the equicontinuity of $(D v_\eps)$, we obtain a subsequence $(v_\eps)$ converging to a viscosity solution $v$ of equation \eqref{eigennormpl} in $C^{1,\alpha}_\text{\text{loc}}(\Omega')$ when $\eps\rightarrow 0$. 

\noindent For $\lambda>0$ and a given continuous boundary data, uniqueness for viscosity solutions of \eqref{eigennormpl} is easy to prove. By using uniqueness, we conclude in Step 5 that the function $v$ is the unique viscosity  solution to $\eqref{eigennormpl}$ with given boundary data $u$. Since $u$ is a solution to \eqref{eigennormpl}, we get that $u=v$. This gives a proof for Theorem \ref{thm:main2}.\\

\noindent \textbf{Step 1: Local $C^{1,\beta}$ regularity for $v_\eps$}
Let $v_\eps\in W^{1,p}(\Om')$ be a weak solution of the regularized problem \eqref{regutpb}. Since $p-2<p$ and $h_\eps\in L^q(\Om')$ with $q>n/2$, regularity theory implies that the solutions $v_{\eps}$ are  bounded and locally H\"older continuous. This follows from the Sobolev embedding for $p>n$ and from \cite[Theorems 7.1,7.2, Chapter 4 p.\, 286--290]{LU68} for $p\leq  n$. Since $h_\eps \in C(\overline{\Om'})$ is  bounded and the exponent on the gradient in the left term is less than $p$,  we also have $v_{\eps}\in C^{1,\a(\eps)}_{\text{loc}}(\Om')\cap W^{2,2}_{\text{loc}}(\Om')$ (see \cite[Theorem 8.7, chapter 4,  p.\,311]{LU68}, and also \cite{Dib83,Tolk84} for more general regularity results.) This observation is useful, since we will derive estimates for $D v_\eps$ by using test functions involving the derivatives of  $v_{\eps}$.\\

\noindent\textbf{Step 2: Uniform boundedness of $\norm{D v_{\eps}}_{L^p(\Om')}$ and  $\norm{ v_{\eps}}_{L^\infty(\Om')}$}
First we derive  a uniform bound for $\norm{D v_{\eps}}_{L^p(\Om')}$. Considering the weak formulation and taking $\varphi= v_{\eps}-u$ as a test function, we have
\begin{align*}
\int_{\Omega'}\left( |D v_{\eps}|^2+\eps^2\right)^{\frac{p-2}{2}} |D v_{\eps}|^2\,dx &\leq
\int_{\Omega'}(|D v_{\eps}|+\eps)^{p-2}|v_{\eps}-u||h_\eps|\, dx\\
&\quad +\int_{\Omega'}\left( |D v_{\eps}|^2+\eps^2\right)^{\frac{p-2}{2}}|D v_{\eps}\cdot D u|\,dx\\
&\quad +\lambda \int_{\Omega'} (|D v_{\eps}|+\eps)^{p-2}|v_\eps||v_{\eps}-u|\, dx\\
&\leq
\int_{\Omega'}(|D v_{\eps}|+\eps)^{p-2}|v_{\eps}-u||h_\eps|\, dx\\
&\quad +\int_{\Omega'}\left( |D v_{\eps}|^2+\eps^2\right)^{\frac{p-2}{2}}|D v_{\eps}||D u|\,dx\\
&\quad +\lambda \int_{\Omega'} (|D v_{\eps}|+\eps)^{p-2}|v_{\eps}-u|^2\, dx\\
&\quad +\lambda \int_{\Omega'} (|D v_{\eps}|+\eps)^{p-2}|v_{\eps}-u||u|\, dx.
\end{align*}
Using the inequality
$$\int_{\Omega'} |D v_{\eps}|^p\,dx\leq\int_{\Omega'}\left( |D v_{\eps}|^2+\eps^2\right)^{\frac{p-2}{2}} |D v_{\eps}|^2\,dx$$
together with Young's inequality and the previous estimate, we obtain
\begin{align}\label{com1}
\int_{\Omega'} |D v_{\eps}|^p\,dx &\leq 
\delta_0\int_{\Omega'} |D v_{\eps}|^{p}\, dx
 + C(p)\eps^p |\Omega'|+\int_{\Omega'}| D u|^p\,dx\nonumber\\
 &\quad+C(p)\int_{\Om}|v_{\eps}-u|^{p/2}|h_\eps|^{p/2}\, dx\nonumber\\
&\quad+\lambda C(p)\int_{\Omega'} |v_{\eps}-u|^p\, dx+C(p)\lambda \int_{\Omega'} |u|^p\, dx.
\end{align}
If $\lambda=\lambda(p,n, \Om')>0$ is small enough, then using the Sobolev embedding, we get 
\begin{eqnarray}
\int_{\Omega'} |D v_{\eps}|^p\,dx &\leq&
\delta_1\int_{\Omega'} |D v_{\eps}|^{p}\, dx+ C(p)\int_{\Om'}|v_{\eps}-u|^{p/2}|h_\eps|^{p/2}\, dx\nonumber\\
&\quad& +\delta_2\int_{\Omega'} |D v_{\eps}|^{p}\, dx+ C(p,n) \int_{\Om'}|D u|^p\,dx\\
&\quad&+C(p)\lambda \int_{\Omega'} |u|^p\, dx+C(p)\eps^p |\Omega'|.
\end{eqnarray}
Now we have to estimate $\int_{\Om'}|v_{\eps}-u|^{p/2}|h_\eps|^{p/2}\, dx$. 
We deal separately with the cases $p<n$, $p=n$ and  $p>n$.

\subsubsection*{Case $p < n$} We denote by $p^* = \dfrac{np}{n-p}$ the Sobolev's conjugate exponent of $p$. Using
Sobolev's and H\"older's inequalities and noticing that $\frac{p^*p}{2p^*-p}=\frac{np}{n+p}$, we get
\begin{eqnarray}\label{com2}
\int_{\Omega'}|v_{\eps}-u|^{\frac{p}{2}}|h_\eps|^{\frac{p}{2}}\, dx&\leq&
\norm{v_{\eps}-u}_{L^{p^*}(\Om')}^{\frac{p}{2}}
\left(\int_{\Om'}|h_\eps|^{\frac{p^*p}{2p^*-p}}\, dx\right)^{\frac{2p^*-p}{2p^*}}\nonumber\\
&\leq& C(p,n,|\Om'|)\norm{Dv_{\eps}-Du}_{L^{p}(\Om')}^{\frac{p}{2}}
\norm{h_\eps}_{L^{\frac{np}{n+p}}(\Om')}^{p/2}\nonumber\\
&\leq& \delta_3 \int_{\Omega'} |Dv_{\eps}-Du|^{p}\, dx+C(p,n,|\Om'|)\norm{h_\eps}_{L^{\frac{np}{n+p}}(\Om')}^{p}\nonumber\\
&\leq& \delta_4\int_{\Omega'} |Dv_{\eps}|^{p}\, dx+C(p,n,|\Om'|)\norm{Du}_{L^p(\Om')}^p\nonumber\\
&\quad& +C(p,n,|\Om'|)\norm{h_\eps}_{L^{\frac{np}{n+p}}(\Om')}^{p}.
\end{eqnarray}
Combining \eqref{com1} and \eqref{com2} and choosing $\delta_1+\delta_2+C(p)\delta_4=1/2$ in order to  absorb terms, we obtain  remembering the definition of the function $h_\eps$
\begin{align}
\norm{D v_{\eps}}_{L^{p}(\Om')}^p&\leq C(p,n,|\Om'|)\left(\norm{h_\eps}_{L^{\frac{np}{n+p}}(\Om')}^p+\int_{\Om'} (|D u|+1+ |u|)^p\,dx\right)\nonumber\\
&\leq  C(p,n,|\Om'|)\left(\norm{f}_{L^{\frac{np}{n+p}}(\Om')}^p+|\Om'|^{1+p/n}\norm{u}_{L^\infty (\Om)}^{p}\right)\nonumber\\
&\quad+C(p,n, |\Om'|)\int_{\Om'} (|D u|+1+ |u|)^p\,dx.\label{cacppet}
\end{align}

\subsubsection*{Case $p=n$}
First we calculate
\begin{align}\label{comet3}
\int_{\Omega'} |v_{\eps}-u|^{p/2}|h_\eps|^{p/2}\, dx&\leq\delta_5\norm{v_{\eps}-u}_{L^{p}(\Om')}^p+C(p)\norm{h_\eps}_{L^p(\Om')}^p\nonumber\\
&\leq\delta_5C(p,n,|\Om'|)\norm{Dv_{\eps}-Du}_{L^{p}(\Om')}^p+C(p)\norm{h_\eps}_{L^p(\Om')}^p\nonumber\\
&\leq\delta_6 \norm{Dv_{\eps}}_{L^{p}(\Om')}^p+ C(n,p,|\Om'|)\norm{Du}_{L^p(\Om')}^p\nonumber \\
&\quad+C(p)\norm{h_\eps}_{L^p(\Om')}^p.
\end{align}
Combing \eqref{com1} and \eqref{comet3} and choosing $\delta_1+\delta_2+C(p)\delta_6=1/2$, we obtain
\begin{align}\label{cacpn}
\norm{D v_{\eps}}_{L^{p}(\Om')}^p&\leq C(p,n,|\Om'|)\norm{h_\eps}_{L^{p}(\Om')}^p+C(p,n, \Om')\int_{\Om'} (|D u|+|u|+1)^p\,dx\nonumber\\
&\leq C(p,n,|\Om'|)\left(\norm{f}_{L^{n}(\Om')}^p+|\Om'|\norm{u}_{L^\infty(\Om')}^p\right)\nonumber\\
&\quad +C(p,n, \Om')\int_{\Om'} (|D u|+|u|+1)^p\,dx.
\end{align}

\subsubsection*{Case $p>n$}
First we calculate
\begin{align}
\int_{\Omega'} |v_{\eps}-u|^{p/2}|h_\eps|^{p/2}\, dx&\leq\delta_7
\norm{v_{\eps}-u}_{L^{\infty}(\Om')}^p+
C(p,n)\norm{h_\eps}_{L^{\frac{p}{2}}(\Om')}^{p}\nonumber\\
&\leq\delta_7 C(p,n,|\Om'|)\norm{Dv_{\eps}-Du}_{L^{p}(\Om')}^p\nonumber\\
&\quad +
C(p,n)\norm{h_\eps}_{L^{\frac{p}{2}}(\Om')}^p\nonumber\\
&\leq\delta_8 \norm{Dv_{\eps}}_{L^{p}(\Om')}^p+ C(p,n)\norm{h_\eps}_{L^{\frac{p}{2}}(\Om')}^p\nonumber\\
&\quad +C(p,n,|\Om'|)\norm{Du}_{L^p(\Om')}^p\label{com4}.
\end{align}
Combing \eqref{com1} and \eqref{com4} and choosing $\delta_1+\delta_2+C(p)\delta_8=1/2$, we get
\begin{align}\label{cacpg}
\norm{D v_{\eps}}_{L^{p}(\Om')}^p&\leq C(p,n,|\Om'|)\norm{h_\eps}_{L^{\frac{p}{2}}(\Om')}^p+C(p,n, |\Om'|)\int_{\Om'} (|D u|+1+|u|)^p\,dx\nonumber\\
 &\leq C(p,n,|\Om'|)\left(\norm{f}_{L^{\frac{p}{2}}(\Om')}^p+|\Om'|^2\norm{u}_{L^\infty(\Om')}^p\right)\nonumber\\
 &\quad+C(p,n, |\Om'|)\int_{\Om'} (|D u|+1+|u|)^p\,dx.
\end{align}
Once the boundedness of $\norm{Dv_\eps}_{L^p(\Om')}$ is proved, we can derive a uniform bound for $\norm{v_\eps}_{L^\infty(\Om')}$. Using the Sobolev embedding, in the case $p>n$ we get 
\begin{align*}
\norm{v_\eps}_{L^\infty (\Omega')}&\leq \norm{v_\eps-u}_{L^\infty(\Om')}+\norm{u}_{L^\infty(\Om')}\nonumber\\
&\leq C(n, \Om', p)\norm{D v_\eps-D u}_{L^p(\Om')}+\norm{u}_{L^\infty(\Om')}\\
&\leq C(p,n,|\Om'|)\left(\norm{f}_{L^{q}(\Om')}+\norm{u}_{W^{1,p}(\Om')}+ \norm{u}_{L^\infty(\Om')}+1\right).\nonumber
\end{align*}
For $p\leq n$, since $h_\eps\in L^q(\Om)$ for $q>\frac n2$, we can apply \cite[Theorem 7.1, Chapter 4]{LU68} giving an estimate for  $\norm{v_\eps}_{L^\infty(\Om')}$ when combined with the previous estimates of $\norm{Dv_\eps}_{L^p(\Om')}$. We get 
\begin{align*}
\norm{v_\eps}_{L^\infty (\Omega')}&\leq C\left( \norm{u}_{L^\infty(\Om)},p, n, |\Om'|, \norm{h_\eps}_{L^{q}(\Om')}, \norm{v_\eps}_{L^{p^*}(\Om')}\right)\nonumber\\
&\leq C\left( \norm{u}_{L^\infty(\Om')},p, n,q, |\Om'|, \norm{f}_{L^{q}(\Om')}, \norm{u}_{W^{1,p}(\Om')}\right),
\end{align*} 
where we also used the estimate 
\begin{align*}
\norm{v_\eps}_{L^{p^*}(\Om')}&\leq \norm{v_\eps-u}_{L^{p^*}(\Om')}+\norm{u}_{L^{p^*}(\Om')}\\
&\leq C(p,n, \Om')(\norm{u}_{L^\infty(\Om')}+\norm{u}_{W^{1,p}(\Om')}+\norm{v_\eps}_{W^{1,p}(\Om')}).
\end{align*}
In both cases $p\leq n$ and $p>n$, by using the estimate \eqref{metrorome} we get  
\begin{align}\label{lorop}
\norm{v_\eps}_{L^\infty (\Omega')}
&\leq C\left( \norm{u}_{L^\infty(\Om')},p, n,q, |\Om'|, \norm{f}_{L^{q}(\Om')}\right).
\end{align} 

\noindent\textbf{Step 3: Local uniform  Lipschitz estimate for $v_{\eps}$}\label{dmcheck}
In this subsection we derive a uniform local  gradient estimate for $v_{\eps}$ by combining \cite[Theorem 1.5]{DM2010} with the previous estimates \eqref{cacppet}-\eqref{lorop}. We follow the main steps of Duzaar and Mingione \cite{DM2010}. For the sake of completeness, we give some details of these steps.
We denote by $V(x):= h_\eps(x)-\lambda v_{\eps} (x)$. 
Then $v_\eps$ solves the equation 
\begin{equation*}
\left\{\begin{array}{ll}
 -\mathrm{div} \left(\left( | D v_{\eps}|^2+ \eps^2\right)^{(p-2)/2}  D v_{\eps} \right)=(|D v_{\eps}|^{2}+\eps^2)^{\frac{p-2}{2}}V,  & x\in \Omega',\\
v_{\eps}=u &x\in\partial\Om'.
\end{array}
\right.
\end{equation*}
The Duzaar-Mingione gradient estimate relies on the use of a  nonlinear potential of the function $|V|^2$ defined by 
 \begin{equation}
\mathcal{P}^V(x, R):=\int_0^R\left(\dfrac{|V|^2(B(x, \rho))}{\rho^{n-2}}\right)^{\frac{1}{2}}\, \dfrac{d\rho}{\rho},
\end{equation}
where 
$$|V|^2(B(x, \rho)):=\int_{B(x, \rho)} |V(y)|^2\, dy.$$

 Let us recall the main ingredients of the proof of the result of \cite{DM2010}.  A key step is to derive a Caccioppoli type estimate for the function $ (|Dv_{\eps}|^{2}+\eps^2)^{\frac{p}{2}}$ with a suitable remainder involving $|V|^2$. Relying on the regularity result of Step 1, this can be done by taking
\[
\varphi_{ij}(x):=\frac{\partial}{\partial x_j}\left(\eta(x)^2\left((|Dv_{\eps}(x)|^{2}+\eps^2)^{\frac{p}{2}}-k\right)_+ \frac{\partial v_\eps(x)}{\partial x_i}\right) 
\]
as test functions in the weak formulation, where $\eta$ is a non negative cut-off function. Next,  a modification of the De Giorgi techniques allowed them to  get pointwise estimate of $|Dv_{\eps}|^p$  in terms of the $L^{2p}$ norm of $Dv_{\eps}$ and the nonlinear potential $\mathcal{P}^{V}$. Finally, using interpolation, they improved the estimate in terms of the natural $L^p$ norm of the gradient and the $L^\infty$ norm of the nonlinear potential. 

Our approximation is slightly different, but the Caccioppoli type estimate of \cite[Lemma 3.1]{DM2010} (adapted for the new right hand side) holds for $2<p\leq n$ and also for $p>n$. Indeed, by using the weak formulation with the test function $\varphi_{ij}$ and integration by parts, there exists a constant $C=C(p,n)$ such that for any ball $B_R:=B(x,R)\subset \Om'$,
\begin{eqnarray*}
\int_{B_{\frac{R}{2}}}\left|D\left((|Dv_{\eps}|^{2}+\eps^2)^{\frac{p}{2}}-k\right)_+\right|^2\, dy\leq\dfrac{C}{R^2}\int_{B_R} \left((|Dv_{\eps}|^{2}+\eps^2)^{\frac{p}{2}}-k\right)^2_+\,dy\\
+C\int_{B_R}\left|\left(\eps^2+\norm{D v_{\eps}}_{L^{\infty}(B_R)}^2\right)^{(p-1)/2} V\right|^2\, dy.
\end{eqnarray*}
It follows that the oscillation improvement estimate \cite[Lemma 3.2]{DM2010} holds. Once we have such control on the level sets of $|D v_\eps|^p$, a standard modification of the De Giorgi iteration argument implies the following potential estimate (see for example \cite[Lemma 3.3]{DM2010}) 
\begin{eqnarray*}
\left(|Dv_{\eps}(x)|^{2}+\eps^2\right)^{\frac{p}{2}}&\leq& C\left(
\vint_{B_R} \left(|Dv_{\eps}|^{2}+\eps^2\right)^{p}\,dy\right)^{1/2}\\
&+&C\left(\eps^2+\norm{Dv_{\eps}}_{L^{\infty}(B_R)}^2\right)^{\frac{
p-1}{2}} \mathcal{P}^{V}(x, R),
\end{eqnarray*}
where $C=C(p,n)$. Proceeding as in  \cite{DM2010} we get for $R/2<\rho<r<R$,
\begin{align*}
\left(\norm{Dv_{\eps}}_{L^{\infty}(B_{\rho})}^{2}+\eps^2\right)^{\frac{p}{2}}&\leq C \dfrac{\left(\eps^2+\norm{Dv_{\eps}}_{L^{\infty}(B_r)}^2\right)^{\frac{p}{4}}}{(r-\rho)^{n/2}}
\left(\displaystyle\int_{B_r}\! \left(|Dv_{\eps}|^{2}+\eps^2\right)^{p/2}\,dy\right)^{1/2}\\
&\quad+C\left(\eps^2+\norm{Dv_{\eps}}_{L^{\infty}(B_r)}^2\right)^{\frac{p-1}{2}} \norm{\mathcal{P}^{V}(\cdot, R)}_{L^{\infty}(B_R)}\\
&\leq\dfrac{1}{2} \left(\eps^2+\norm{Dv_{\eps}}_{L^{\infty}(B_r)}^2\right)^{\frac{p}{2}} +C \norm{\mathcal{P}^{V}(\cdot, R)}_{L^{\infty}(B_R)}^{p}\\
&\quad+\dfrac{C}{(r-\rho)^{n}}\displaystyle\int_{B_r}\! \left(|Dv_{\eps}|^{2}+\eps^2\right)^{p/2}\,dy,
\end{align*}
where $C=C(p,n)$. Now the standard iteration lemma (see for example \cite[Lemma 2.1]{DM2010}) implies that
\begin{eqnarray}\label{maija}
\left(\norm{Dv_{\eps}}_{L^{\infty}(B_{R/2})}^{2}+\eps^2\right)^{\frac{p}{2}}&\leq C\displaystyle\vint_{B_R}\! \left(|Dv_{\eps}|^{2}+\eps^2\right)^{p/2}\,dy\nonumber\\
&+C \norm{\mathcal{P}^{V}(\cdot, R)}_{L^{\infty}(B_R)}^p,
\end{eqnarray}
where $C=C(p,n)$. Consequently, combining \eqref{cacppet}, \eqref{cacpn}, \eqref{cacpg} and \eqref{maija} we get 
\begin{eqnarray}
\norm{Dv_{\eps}}_{L^{\infty}(B_{R/2})}&\leq&C\left( R^{-n/p}\norm{Dv_{\eps}}_{L^p(B_R)}+ \norm{\mathcal{P}^{V}(\cdot, R)}_{L^{\infty}(B_R)}+1\right),
\nonumber
\end{eqnarray}
for all $R$ such that $B_R\subset\Om'$ and where  $C=C(p,n)$. Since $v_\eps$ is uniformly bounded in $L^\infty (\Om')$ and $h_\eps$ is uniformly bounded in $L^q(\Om')$, we have  $V\in L^q(\Om')$. We obtain
$$\int_{B(x, \rho)} |V(y)|^2\, dy\leq\norm{V}_{L^q(\Om')}^2 |B(x, \rho)|^{\frac{q-2}{q}}\leq C \norm{V}_{L^q(\Om')}^2 \rho^{\frac{n(q-2)}{q}},$$
where $C=C(n)$, and 
$$\mathcal{P}^{V}(x, R)\leq\norm{V}_{L^q(\Om')}\int_0^R \rho^{\frac{n(q-2)}{2q}-\frac{n}{2}}\, d\rho\leq C R^{\frac{q-n}{q}},$$
where $C=C(q,n) \norm{V}_{L^q(\Om')}$. It follows that
\begin{equation}\label{potf}
\underset{B(x,R)}{\sup}\, \mathcal{P}^{V}(\cdot, R)\leq C\underset{B(x,R)}{\sup} R^{\frac{q-n}{q}}<\infty,
\end{equation}
where $C=C(n, q,\norm{V}_{L^q(\Om')})$. 
Recalling that $V=h_\eps-\lambda v_\eps$, and using the bound \eqref{lorop} for $\norm{v_\eps}_{L^\infty(\Om')}$, we get 
\begin{equation}
\norm{V}_{L^q(\Om')}\leq C\left(p,n,q, |\Om'|,\norm{f}_{L^q(\Om')},\norm{u}_{L^\infty(\Om')}\right).
\end{equation}
Hence, 
\[
\norm{Dv_{\eps}}_{L^{\infty}(B_{R/2})}\leq \tilde C\left(p,n,\Om, q,\norm{f}_{L^q(\Om')},\norm{u}_{L^\infty(\Om')}, R\right). 
\]

\noindent\textbf{Step 4: Local uniform $C^{1,\beta}$ estimate for $u_{\eps}$}
Since  $Dv_{\eps}$ is locally uniformly  bounded in $L^{\infty}$ with respect to $\eps$, the function 
$$\mu_{\eps}:=(|Dv_{\eps}|^{2}+\eps^2)^{\frac{p-2}{2}}V$$
 is also locally bounded in $L^{q}$ with $q>n$ and satisfies
\begin{eqnarray*}
\int_{B_r(x)}|\mu_\eps|dy&\leq& C(p)\left(\norm{D v_{\eps}}_{L^{\infty}(B_r(x))}^{p-2}+1\right)\int_{B_r(x)}|V(y)|\, dy\\
&\leq& C(p,n)\left(\norm{D v_{\eps}}_{L^{\infty}(B_r(x))}^{p-2}+1\right)\norm{V}_{L^q(\Om')}r^{\frac{n(q-1)}{q}}\\
&\leq& \tilde C\left(q,n,p, \Om', \norm{f}_{L^q(\Omega')}, \norm{u}_{L^\infty(\Omega')}\right) r^{n-p+\delta},
\end{eqnarray*}
where $\delta=\frac{qp-n}{q}$, $\delta\in(p-1,p)$.
Applying the result of Lieberman \cite[Theorem 5.3]{Lie93} ($(v_{\eps})$ being also  bounded in $L^{\infty}$), we get that  $v_{\eps}$ are locally of class $C^{1, \beta}$  for some $\beta=\beta(p,q,n)$ and  for any $\Om''\subset\subset\Om'$
\begin{equation}\label{fabi}
[v_{\eps}]_{C^{1,\beta}(\Omega'')}\leq C=C\left(p,q,n, |\Om'|, \norm{u}_{L^\infty(\Omega')}, d'', \norm{f}_{L^q(\Om')}\right),
\end{equation}
where $d''=\dist(\Om'', \partial\Om')$.

\noindent\textbf{Step 5: Convergence in the weak and viscosity sense and conclusion}\label{step5}
We get from \eqref{fabi} and the Arzel\`a-Ascoli theorem that $(u_{\eps})$ converges (up to a subsequence) to a function $v$ in $C^{1,\a}_{\text{loc}}(\Om')$ for some $\a=\a(q,p,n)<\beta$. Passing to the limit within the weak formulation, $v$ is a weak solution to
\begin{equation}\label{berty}
-\Delta_p v=|Dv|^{p-2}(h-\lambda v),
\end{equation}
see Appendix \ref{unif limit function} for details. 
Passing to the limit in \eqref{fabi},  we get that for any $\Om''\subset\subset\Om'$, we have the estimate
\begin{align*}
\norm{v}_{C^{1, \alpha}(\Om'')}\leq C\left(p,n,q,d'', |\Om'|, \norm{u}_{L^\infty(\Om')}, \norm{f}_{L^q(\Om')}\right).
\end{align*}
From the boundedness of $v_\eps$, it follows that $v$ is a bounded weak solution of the Dirichlet problem associated to  \eqref{berty}. Since $(v_\eps-u)$ is uniformly bounded in $W^{1,p}_0(\Om')$, we have $(v-u)\in W^{1,p}_0(\Om')$. Assuming sufficient regularity for the boundary $\partial \Om'$, we have $v\in C(\overline{\Om'})$ and for  any $x_0\in\partial\Omega'$ $\underset{x\to x_0}{\lim}\, v(x)=u(x_0)$. The reader can find further discussion of the boundary regularity problem for elliptic equations in the monograph of Mal\'y and Ziemer \cite{malyzim}.
On the other hand, the local H\"older continuity of $Dv_{\eps}$ and the H\"older continuity of $h_\eps$ imply, by the classical elliptic regularity theory, that $v_{\eps}$ is also a classical solution to
$$ -\mathrm{div} \left(\left( | D v_{\eps}|^2+ \eps^2\right)^{(p-2)/2}  D v_{\eps} \right)=(|Dv_{\eps}|^{2}+\eps^2)^{\frac{p-2}{2}} (h_\eps-\lambda v_\eps)\quad\text{in}\,\, \Om'.$$
This implies that  $v_{\eps}$ solves in the classical sense 
\begin{equation}\label{perti}
-\Delta v_{\eps}-(p-2) \frac{D^2v_{\eps} Dv_{\eps}\cdot Dv_{\eps}}{| D v_{\eps}|^2+ \eps^2} =h_\eps-\lambda v_\eps \qquad\text{in}\,\,\Om'.
\end{equation}
	Hence $v_{\eps}$ is a continuous  viscosity solution of the  Dirichlet problem associated
to equation \eqref{perti} with continuous  boundary data $u$. 
Passing to the limit in \eqref{perti}, we get  that the limit function $v$ is also a continuous viscosity solution of \eqref{eigennormpl} with boundary data equals $u$, see Appendix \ref{viscconvergence}.  The viscosity solution to \eqref{eigennormpl} is understood in the sense of Definition \ref{defweak}. It is easy to see that the  fixed viscosity solution $u$ of \eqref{normpl} is a viscosity solution to \eqref{eigennormpl} with the weaker Definition \ref{defweak} ($\eta$ is then taken as an eigen-vector of $D^2\phi(x_0)$).
It follows (see the Appendix \ref{unieigen}  for details) that, for a given boundary data, the Dirichlet problem associated to \eqref{eigennormpl} admits a unique  viscosity solution.
By uniqueness, we conclude that the limit function $v$ is the unique viscosity solution of \eqref{eigennormpl} and since $u$ is a viscosity solution to this problem, we conclude that $u=v$ in $\Om'$.
It follows that  $u$ is  of class $C^{1,\a}_{\text{loc}}$ for some $\a=\a(p,q,n)$ and  the estimate of Theorem \ref{thm:main2} holds.

\section{Nearly optimal H\"older exponent for gradients}\label{chpt6}
In this section we prove Theorem \ref{thm:main3}. Assume that $f\in L^q(\Omega) \cap C(\Omega)$ and fix arbitrary $\xi>0$. We will prove that the viscosity solutions to \eqref{normpl} are of class $C^{1,\alpha_\xi}_{\text{\text{loc}}}$, where 
\begin{equation*}
\alpha_\xi=\left\{\begin{array}{ll}
 \alpha_0-\xi  & \text{when}\ q=\infty,\\
\min(\alpha_0-\xi,1-\frac nq) & \text{when}\ \max(n,\frac p2,2)<q<\infty,
\end{array}
\right.
\end{equation*}
and $\alpha_0$ is the optimal H\"older exponent in an a priori estimate for gradients of $p$-harmonic functions. In the case $q=\infty$ we only assume that $p>1$, whereas in the case $q<\infty$ we require $p>2$.

The question of optimal regularity for inhomogeneous $p$-Laplacian in divergence form has received attention as well, see \cite{lindgren2013regularity,kuusi2014guide,aruijotu,arauijoz}. An alternative approach to study optimal regularity questions for $p$-Poisson problem in divergence form could be based on \cite[equation (1.38)]{kuusim12}. In our paper we do not try to quantify the explicit optimal value of $\alpha$ in $C^{1,\alpha}$ estimate to the homogenous case.
\begin{remark}
If $p\geq 2$ and $f$ is a continuous and bounded function, in  the case that $\Om$ is either a ball or an annulus, radial viscosity solutions to \eqref{normpl} have a better regularity and they are in $C^{1,1}(\Om)$ (see \cite[Theorem 1.1]{birdemen1}).
\end{remark}


\subsection{The case $q=\infty$}\label{optimalsection1}
In this subsection we prove Theorem \ref{thm:main3} when $f\in L^\infty(\Omega)\cap C(\Omega)$. Since our results are local, by translation and rescaling we can restrict our study in the unit ball $B_1\subset \Om$ and show the regularity at $0\in B_1\subset \Omega$. Like previously, it is useful to do suitable rescaling to get an Arzel\`a-Ascoli type compactness lemma. During the rest of this section, for $\delta_0>0$ to be determined later, we assume that $||u||_{L^\infty(B_1)}\leq 1$ and $||f||_{L^\infty(B_1)}\leq \delta_0$ without loss of generality. This can be seen like before: Let $\kappa=(||u||_{L^\infty(B_1)}+\delta_0^{-1}||f||_{L^q(B_1)})^{-1}$. Setting $\tilde{u}=\kappa u$, then $\tilde{u}$ satisfies
\begin{equation*}
-\Delta_p^N\left(\tilde{u}\right)=\tilde{f}
\end{equation*}
with $||\tilde{u}||_{L^\infty(B_1)}\leq 1$ and $||\tilde{f}||_{L^q(B_1)}\le \delta_0$.

For convenience, in this subsection we denote by $C$ different constants depending only on $p$ and $n$.

First we use our regularity result from Section \ref{chpt3} to show that the solutions to \eqref{normpl} can be approximated by $p$-harmonic functions in $C^{1,\alpha}_{\text{loc}}$ for some small $\alpha>0$. 

\begin{lemma}
Let $u\in C(B_1)$ be a viscosity solution to equation \eqref{normpl}. For given $\eps>0$, there exists $\delta_0=\delta_0(p,n,\eps)$ such that for $||u||_{L^\infty(B_1)}\leq 1$, $||f||_{L^\infty(B_1)}\leq \delta_0$, there exists a $p$-harmonic function $h$ in $B_{3/4}$ satisfying
\[
||u-h||_{L^\infty(B_{1/2})}< \eps\ \ \ \text{and}\ \ \ ||D u-D h||_{L^\infty(B_{1/2})}< \eps.
\]
\end{lemma}

\begin{proof}
Suppose that the lemma is not true. Then, for some $\eps_0>0$ there is a uniformly bounded sequence of continuous functions $(u_j)$ and a sequence $(f_j)\subset C(\Omega)\cap L^\infty(\Omega)$, $||f_j||_{L^\infty(B_1)}\rightarrow 0$, such that
\[
-\Delta^N_p u_j=f_j,
\]
but for all $p$-harmonic functions $h$ defined in $B_{3/4}$ we have either $||u_j-h||_{L^\infty(B_{1/2})}\geq \eps_0$ or $||D u_j-D h||_{L^\infty(B_{1/2})}\geq \eps_0$.

By Theorem \ref{thm:main1}, $(u_j)\subset C^{1,\alpha}(B_{3/4})$ for some $\alpha>0$, so by the Arzel\`a-Ascoli theorem there is a subsequence, still denoted by $(u_j)$, which converges to some function $h$ in $C^{1,\alpha}(B_{1/2})$. Then the limit function $h$ satisfies $\Delta^N_p h=0$ in the viscosity sense, so it also satisfies $\Delta_p h=0$ in the weak sense. By $C^{1,\alpha}$ convergence, there is $j_0\in \N$ such that $||u_{j_0}-h||_{L^\infty(B_{1/2})}< \eps_0$ and $||D u_{j_0}-D h||_{L^\infty(B_{1/2})}< \eps_0$. We have reached a contradiction.
\end{proof}

By using the approximation with $p$-harmonic functions, in the next lemma we obtain an oscillation estimate for solutions $u$ to \eqref{normpl} near the critical set $\{x\ :\ D u(x)=0\}$.

\begin{lemma}\label{applemma}
There exist $\lambda_0=\lambda_0(p,n)\in (0,\frac12)$ and $\delta_0>0$ such that if $||f||_{L^\infty(B_1)}\leq \delta_0$ and $u\in C^{1,\alpha}(B_1)$ is a viscosity solution to \eqref{normpl} in $B_1$ with $||u||_{L^\infty(B_1)}\leq 1$, then
\[
\sup_{x\in B_{\lambda_0}}|u(x)-u(0)|\leq \lambda^{1+\alpha_\xi}_0+|D u(0)|\lambda_0.
\] 
\end{lemma}

\begin{proof}
Take the approximating $p$-harmonic function $h$ from the previous lemma. By the a priori estimate for $p$-harmonic functions, there exist $\lambda_0=\lambda_0(p,n)\in (0,\frac12)$ such that
\[
\sup_{x\in B_{\lambda_0}}|h(x)-[h(0)+D h(0)\cdot x]|\leq C\lambda_0^{1+\alpha_0},
\]
and $C\lambda_0^{1+\alpha_0}\leq \frac12 \lambda_0^{1+\alpha_\xi}$. Now we choose $\eps>0$ satisfying $\eps<\frac16 \lambda_0^{1+\alpha_\xi}$. This $\eps$ determines $\delta_0$ through the previous lemma. We get for all $x\in B_{\lambda_0}$,
\begin{align*}
|u(x)-[u(0)+D u(0)\cdot x]|& \leq |h(x)-[h(0)+D h(0)\cdot x]|\\
& +|(u-h)(x)|+|(u-h)(0)|+|D (u-h)(0)\cdot x|\\
& \leq C\lambda_0^{1+\alpha_0}+3\eps\\
& \leq \lambda_0^{1+\alpha_\xi}.
\end{align*}
The result follows by the triangle inequality.
\end{proof}

Next we iterate the previous estimate to control the oscillation of the solutions in dyadic balls.

\begin{theorem}\label{smallgradient}
Under the assumptions of the previous lemma, there exists a constant $C$ such that
\[
\sup_{x\in B_r}|u(x)-u(0)|\leq C r^{1+\alpha_\xi}\left(1+|D u(0)|r^{-\alpha_\xi}\right)
\]
for all sufficiently small $r\in(0,1)$.
\end{theorem}

\begin{proof}
For $k\in \N$, consider the rescaled function defined in $B_1$,
\[
v_k(x)=\frac{u(\lambda^k_0 x)-u(0)}{\lambda^{k(1+\alpha_\xi)}_0+\sum^{k-1}_{j=0}|D u(0)|\lambda^{k+j\alpha_\xi}_0}.
\]
We have $v_k(0)=0$,
\[
D v_k(0)=\frac{\lambda^k_0}{\lambda^{k(1+\alpha_\xi)}_0+\sum^{k-1}_{j=0}|D u(0)|\lambda^{k+j\alpha_\xi}_0}D u(0),
\]
and
\[
-\Delta^N_p v_k(x)=\frac{\lambda^{2k}_0}{\lambda^{k(1+\alpha_\xi)}_0+\sum^{k-1}_{j=0}|D u(0)|\lambda^{k+j\alpha_\xi}_0}f(\lambda^k_0 x)\leq |\lambda^{k(1-\alpha_\xi)}_0 f(\lambda^k_0 x)|,
\]
where $|\lambda^{k(1-\alpha_\xi)}_0 f(\lambda^k_0 x)|\leq \delta_0$, since $\lambda^{k(1-\alpha_\xi)}_0\leq 1$. 

Let us show by induction that $||v_k||_{L^\infty(B_1)}\leq 1$. By the previous lemma, this holds for $k=1$, so assume that $||v_j||_{L^\infty(B_1)}\leq 1$ for $j\leq k$. As shown above, the function $v_k$ satisfies the conditions of the previous lemma, so we have
\[
\sup_{x\in B_{\lambda_0}}|v_k(x)-v_k(0)|\leq \lambda^{1+\alpha_\xi}_0+|D v_k(0)|\lambda_0.
\]
Hence,
\begin{align*}
\sup_{x\in B_1}&\frac{|u(\lambda^{k+1}_0 x)-u(0)|}{\lambda^{k(1+\alpha_\xi)}_0+\sum^{k-1}_{j=0}|D u(0)|\lambda^{k+j\alpha_\xi}_0}\\
&\leq \lambda^{1+\alpha_\xi}_0+\frac{\lambda^{k+1}_0}{\lambda^{k(1+\alpha_\xi)}_0+\sum^{k-1}_{j=0}|D u(0)|\lambda^{k+j\alpha_\xi}_0}|D u(0)|,
\end{align*}
which reads
\[
\sup_{x\in B_1}|u(\lambda^{k+1}_0 x)-u(0)|\leq \lambda^{(k+1)(1+\alpha_\xi)}_0+\sum^{k}_{j=0}|D u(0)|\lambda^{k+j\alpha_\xi+1}_0.
\]
This is equivalent to $||v_{k+1}||_{L^\infty(B_1)}\leq 1$, so induction is complete. 

We obtain for arbitrary $k$,
\begin{align*}
\sup_{x\in B_{\lambda^{k+1}_0}}\frac{|u(x)-u(0)|}{\lambda^{(k+1)(1+\alpha_\xi)}_0}&\leq 1+\frac{\sum^{k}_{j=0}|D u(0)|\lambda^{k+j\alpha_\xi+1}_0}{\lambda^{(k+1)(1+\alpha_\xi)}_0}\\
& \leq 1+|D u(0)|\lambda^{-(k+1)\alpha_\xi}_0\sum^{k}_{j=0}\lambda^{j\alpha_\xi}_0\\
& \leq \left(1+\frac{1}{1-\lambda^{\alpha_\xi}_0}\right)\left(1+|D u(0)|\lambda^{-(k+1)\alpha_\xi}_0\right)\\
& = C\left(1+|D u(0)|\lambda^{-(k+1)\alpha_\xi}_0\right).
\end{align*}
Since this holds for all $k\in \N$, we obtain for all sufficiently small $r>0$,
\[
\sup_{x\in B_r}|u(x)-u(0)|\leq C r^{1+\alpha_\xi}\left(1+|D u(0)|r^{-\alpha_\xi}\right).\qedhere
\]
\end{proof}

We are ready to show $C^{1,\alpha_\xi}_{\text{\text{loc}}}$ regularity for solutions to equation \eqref{normpl}. If the gradient $D u(0)$ is very small, we obtain the result from the previous theorem. In the other case the result follows from a more classical reasoning using the regularity theory of uniformly elliptic equations.

\begin{theorem}\label{optimalthm}
Under the assumptions of Lemma \ref{applemma}, we have for all sufficiently small $r\in (0,1)$,
\[
\sup_{x\in B_r}\left|u(x)-[u(0)+D u(0)\cdot x] \right|\leq C r^{1+\alpha_\xi}.
\]
\end{theorem}

\begin{proof}
When $|D u(0)|\leq r^{\alpha_\xi}$, Theorem \ref{smallgradient} gives
\begin{align*}
\sup_{x\in B_r}\left|u(x)-[u(0)+D u(0)\cdot x] \right|&\leq \sup_{x\in B_r}\left|u(x)-u(0)\right|+|D u(0)|r\\
& \leq C r^{1+\alpha_\xi}.
\end{align*}

When $|D u(0)|> r^{\alpha_\xi}$, define $\mu:=\text{min}\,(\frac34,|D u(0)|^{1/\alpha_\xi})$ and use the rescaled function
\[
w(x)=\frac{u(\mu x)-u(0)}{\mu^{1+\alpha_\xi}}.
\]
We have $w(0)=0$, $|D w(0)|\geq 1$, and
\[
-\Delta^N_p w(x)=\frac{\mu^2 f(\mu x)}{\mu^{1+\alpha_\xi}}=\mu^{1-\alpha_\xi}f(\mu x),
\]
where $||\mu^{1-\alpha_\xi}f||_{L^\infty (B_1)}\leq \delta_0$. From Theorem \ref{smallgradient} we obtain
\[
\sup_{x\in B_1} |w(x)|=\sup_{x\in B_\mu}\frac{|u(x)-u(0)|}{\mu^{1+\alpha_\xi}}\leq C\left(1+|D u(0)|\mu^{-\alpha_\xi}\right)=C.
\]
Since $u\in C^{1,\alpha}_{\text{loc}}(B_1)$ for some $\alpha>0$, there exists $\gamma\in(0,\frac12)$ such that
\[
|D w(x)|\geq \frac12\qquad \text{in}\ B_\gamma.
\]
For all $p>1$ $w$ is a viscosity solution to $-\Delta_p w=|D w|^{p-2}\mu^{1-\alpha_\xi}f(\mu x)=:g\in C(B_\gamma)$ in $B_\gamma$, so by \cite{julin2012new} it is a weak solution to the same equation, which also satisfies the conditions of \cite[Theorem  5.2, p.\ 277]{LU68}. Hence, $w\in W^{2,2}(B_\gamma)$, so by the local version of \cite[Lemma 9.16, p\, 241]{gilbargt01}, for arbitrary $\eps>0$ it holds $w\in C^{1,1-\eps}(B_\gamma)$. In particular, $w\in C^{1,\alpha_\xi}(B_\gamma)$. Hence, for all $s\in(0,\frac \gamma2)$, we have
\[
\sup_{x\in B_{s}}\left|w(x)-D w(0)\cdot x\right|\leq C s^{1+\alpha_\xi},
\]
or equivalently,
\[
\sup_{x\in B_{s}}\left|\frac{u(\mu x)-u(0)}{\mu^{1+\alpha_\xi}}-\mu^{-\alpha_\xi}D u(0)\cdot x\right|\leq C s^{1+\alpha_\xi},
\]
and we get
\[
\sup_{x\in B_s}\left|u(\mu x)-[u(0)+D u(0)\cdot (\mu x)]\right|\leq C (\mu s)^{1+\alpha_\xi}.
\]
If $r<\frac{\mu \gamma}{2}$, then the previous estimate gives 
\[
\sup_{x\in B_r}\left|u(x)-[u(0)+D u(0)\cdot x] \right|\leq C r^{1+\alpha_\xi}.
\]
If $r\geq \frac{\mu \gamma}{2}$, noticing that $r<\mu$ and $|D u(0)|\leq C\mu^{\alpha_\xi}$ we obtain
\begin{align*}
\sup_{x\in B_r}\left|u(x)-[u(0)+D u(0)\cdot x] \right|&\leq \sup_{x\in B_\mu} |u(x)-u(0)|+|D u(0)|\mu\\
& \leq C\mu^{1+\alpha_\xi}\\
& \leq C\left(\frac2\gamma \right)^{1+\alpha_\xi}r^{1+\alpha_\xi}\\
& \leq Cr^{1+\alpha_\xi}.\qedhere
\end{align*}
\end{proof}
This theorem completes the proof of Theorem \ref{thm:main3} when $f\in C(\Omega)\cap L^\infty(\Omega)$.

\subsection{The case $f\in C\cap L^q$}
In this subsection we assume that $p>2$ and $f\in C(B_1)\cap L^q(B_1)$, and use Theorem \ref{thm:main2} to show that the solutions to equation \eqref{normpl} are of class $C^{1,\alpha_\xi}_{\text{\text{loc}}}$. As previously, for $\delta_0>0$ to be determined later, we take the assumptions $||u||_{L^\infty(B_1)}\leq 1$ and $||f||_{L^q(B_1)}\leq \delta_0$ without loss of generality. We also denote by $C$ different constants depending only on $p$ and $n$.

We follow the reasoning of the first subsection. First we show that the solutions to equation \eqref{normpl} can be approximated by $p$-harmonic functions in $C^{1,\alpha}_{\text{loc}}$. 

\begin{lemma}
Let $u\in C(B_1)$, $||u||_{L^\infty(B_1)}\leq 1$, be a viscosity solution to equation \eqref{normpl}. Given $\eps>0$, there is $\delta_0=\delta_0(p,n,\eps)$ such that if $||f||_{L^q(B_1)}\leq \delta_0$, there is a $p$-harmonic function $h$ in $B_{3/4}$ satisfying
\[
||u-h||_{L^\infty(B_{1/2})}< \eps\ \ \ \text{and}\ \ \ ||D u-D h||_{L^\infty(B_{1/2})}< \eps.
\]
\end{lemma}

\begin{proof}
Thriving for contradiction, assume that there exists $\eps_0>0$ such that there are sequences $(u_j)$ and $(f_j)$ satisfying $||u_j||_{L^\infty(B_1)}\leq 1$, $f_j\in C(B_1)\cap L^q(B_1)$, $||f_j||_{L^q(B_1)}\rightarrow 0$, and
\[
-\Delta^N_p u_j=f_j,
\]
but for all $p$-harmonic functions $h$ in $B_{3/4}$
\[
||u_j-h||_{L^\infty(B_{1/2})}> \eps_0\ \ \ \text{or}\ \ \ ||D u_j-D h||_{L^\infty(B_{1/2})}> \eps_0. 
\]
Recall from Theorem \ref{equivalence3.4} that $u_j$ is a weak solution to
\[
-\Delta_p u_j=|D u_j|^{p-2} f_j\qquad\text{in}\quad B_1.
\]
From Theorem \ref{thm:main2} we know that $(u_j)\subset C^{1,\alpha}(B_{3/4})$ for some $\alpha>0$, so by the  Arzel\`a-Ascoli theorem, there is a subsequence, still denoted by $(u_j)$, converging in $C^{1,\alpha}(B_{3/4})$ to a function $h$. By Appendix \ref{unif limit function}, $h$ is a $p$-harmonic function. We have reached a contradiction.
\end{proof}

The next lemma follows from the previous approximation result as in the first subsection.

\begin{lemma}\label{taas}
There exists $\lambda_0=\lambda_0(p,n)\in (0,\frac12)$ and $\delta_0>0$ such that if $||f||_{L^q(B_1)}\leq \delta_0$ and $u\in C^{1,\alpha}_{\text{loc}}(B_1)$ is a viscosity solution to \eqref{normpl} in $B_1$ with $||u||_{L^\infty(B_1)}\leq 1$, then
\[
\sup_{x\in B_{\lambda_0}}|u(x)-u(0)|\leq \lambda^{1+\alpha_\xi}_0+|D u(0)|\lambda_0.
\] 
\end{lemma}

\begin{theorem}\label{Lqregul}
Under the assumptions of the previous lemma, we have
\[
\sup_{x\in B_r}|u(x)-u(0)|\leq C r^{1+\alpha_\xi}\left(1+|D u(0)|r^{-\alpha_\xi}\right)
\]
for all sufficiently small $r>0$.
\end{theorem}

\begin{proof}
The proof is similar to the proof of Theorem \ref{smallgradient}. Again we consider the rescaled function
\[
v_k(x)=\frac{u(\lambda^k_0 x)-u(0)}{\lambda^{k(1+\alpha_\xi)}_0+\sum^{k-1}_{j=0}|D u(0)|\lambda^{k+j\alpha_\xi}_0},
\]
and see that $v_k(0)=0$,
\[
D v_k(0)=\frac{\lambda^k_0}{\lambda^{k(1+\alpha_\xi)}_0+\sum^{k-1}_{j=0}|D u(0)|\lambda^{k+j\alpha_\xi}_0}D u(0),
\]
and
$$-\Delta^N_pv_k(x)=
\frac{\lambda_0^{2k}}{\lambda^{k(1+\alpha_\xi)}_0+\sum^{k-1}_{j=0}|D u(0)\lambda^{k+j\alpha_\xi}_0} f(\lambda^k_0 x)=:f_{k}(x).
$$
Since $q(1-\alpha_\xi)-n>0$, we estimate
\begin{align*}
\int_{B_1}|f_{k}(x)|^qdx &\leq\int_{B_1}\left(\lambda^{k(1-\alpha_\xi)}_0 |f(\lambda^k_0 x)|\right)^q dx\\
& =\int_{B_{\lambda^k_0}}\left(\lambda^{k(1-\alpha_\xi)}_0 |f(y|)\right)^q \lambda^{-nk}_0 dy\\
& = \int_{B_{\lambda^k_0}}\lambda^{kq(1-\alpha_\xi)-nk}_0 |f(y)|^q dy\\
& \leq \int_{B_{\lambda^k_0}}|f(y)|^q dy.
\end{align*}
Hence, we have $||f_k||_{L^q (B_1)}\leq \delta_0$. By continuing as in the proof of Theorem \ref{smallgradient}, we get the result.
\end{proof}

\begin{theorem}
Under the assumptions of Lemma \ref{taas}, we have
\[
\sup_{x\in B_r}\left|u(x)-[u(0)+D u(0)\cdot x] \right|\leq C r^{1+\alpha_\xi}
\]
for all sufficiently small $r\in (0,1)$.
\end{theorem}
\begin{proof}
We follow the ideas of the proof of Theorem \ref{optimalthm}. We get the result from Theorem \ref{Lqregul} when $|D u(0)|\leq r^{\alpha_\xi}$. In the case $|D u(0)|> r^{\alpha_\xi}$, define the rescaled function
$w(x)=(u(\mu x)-u(0))/\mu^{1+\alpha_\xi})$, for which $w(0)=0$, $|D w(0)|\geq 1$, and
\[
-\Delta^N_p w(x)=\frac{\mu^2 f(\mu x)}{\mu^{1+\alpha_\xi}}=\mu^{1-\alpha_\xi}f(\mu x)=:f_\mu(x),
\]
where $||f_\mu||_{L^q(B_1)}\leq \delta_0$. From Theorem \ref{Lqregul} we get
\[
\sup_{x\in B_1} |w(x)|=\sup_{x\in B_\mu}\frac{|u(x)-u(0)|}{\mu^{1+\alpha_\xi}}\leq C\left(1+|D u(0)|\mu^{-\alpha_\xi}\right)=C.
\]
Since $u\in C^{1,\alpha}_{\text{loc}}(B_1)$ for some $\alpha>0$, there exists $\gamma\in(0,1/2)$ such that
\[
|D w(x)|\geq \frac12\qquad \text{in}\ B_\gamma.
\]
As explained in the proof of Theorem \ref{optimalthm}, we know that  $w\in C^{1,1-n/q}(B_\gamma)$. Since $\alpha_\xi\leq 1-n/q$, we have $w\in C^{1,\alpha_\xi}(B_\gamma)$. Hence, for all $s\in(0,\frac \gamma2)$, we have
\[
\sup_{x\in B_{s}}\left|w(x)-D w(0)\cdot x\right|\leq C s^{1+\alpha_\xi},
\]
and the rest of the argument follows as in the proof of Theorem \ref{optimalthm}.
\end{proof}

The proof of Theorem \ref{thm:main3} is complete.

\appendix

\section{The limit equation in Lemma \ref{flatle}}
\noindent We prove two convergence results needed in the proof of Lemma \ref{flatle}. Assume that there exist a sequence of continuous functions $(f_j)$ with $||f_j||_{L^\infty(B_1)} \rightarrow 0$, a sequence of vectors $(q_j)$ and a sequence of viscosity solutions $(w_j)$ with $\osc_{B_1}{w_j}\leq 1$ to
\[
-\Delta w_j-(p-2) \left\langle D^2w_j\frac{D w_j+q_j}{\abs{D w_j+q_j}}, \frac{D w_j+q_j}{\abs{D w_j+q_j}}\right\rangle=f_j.
\]

\subsection*{Case 1: $(q_j)$ is bounded}\label{appendixa1}
First we show that if $(q_j)$ is bounded, there is a subsequence $(w_j)$ converging to a limit $w_{\infty}$, which satisfies
\begin{equation}\label{normaali}
-\tr\left(\left(I+(p-2)\dfrac{Dw_\infty+q_\infty}{|Dw_\infty+q_\infty|}\otimes\dfrac{Dw_\infty+q_\infty}{|Dw_\infty+q_\infty|}\right )D^2 w_\infty\right) =0   \quad\text{in}\, B_1
\end{equation}
in a viscosity sense. Here $q_j\rightarrow q_\infty$ up to the same subsequence.
We show that $w_{\infty}$ is a subsolution of \eqref{normaali} (the case of supersolution being similar).
We  fix $\phi\in C^2(\Omega)$ such that $w_{\infty}-\phi$  has a strict maximum at $x_0$. As $w_{\infty}$ is the uniform limit of the subsequence $(w_j)$ and $x_0$ is a strict maximum point, there
exists a sequence of points $x_j\rightarrow x_0$ such that $(w_j-\phi)$ has a local maximum at $x_j$.

Suppose first that $-D\phi(x_0)\neq q_\infty$. Then $-D\phi(x_j)\neq q_j$ when $j$ is large, and at those points we have
\[
-\Delta \phi_j-(p-2) \left\langle D^2\phi_j\frac{D \phi_j+q_j}{\abs{D \phi_j+q_j}}, \frac{D \phi_j+q_j}{\abs{D \phi_j+q_j}}\right\rangle\leq f_j.
\]
Passing to the limit, we get the desired result.

Suppose next that $-D\phi(x_0)=q_\infty$. We have to consider two cases. Assuming
that there exists a subsequence still indexed by $j$ such that $|D\phi(x_j)+q_j| > 0$ for all $j$ in the
subsequence, then
\[
-\Delta \phi_j-(p-2) \left\langle D^2\phi_j\frac{D \phi_j+q_j}{\abs{D \phi_j+q_j}}, \frac{D \phi_j+q_j}{\abs{D \phi_j+q_j}}\right\rangle\leq f_j,
\]
and we conclude by passing to the limit. If such a subsequence does not exist, then we have
$$ -\Delta \phi(x_j)-(p-2)\lambda_{max}(D^2\phi(x_j))\leq f_j(x_j)$$
for $j$ large enough.
Passing to the limit we get
$$-\Delta \phi(x_0)-(p-2)\lambda_{max}(D^2\phi(x_0))\leq 0.$$
We have shown the desired result.

\subsection*{Case 2: $(q_j)$ is unbounded}\label{appendixa2}
When $(q_j)$ is unbounded, take a subsequence, still denoted by $(q_j)$, for which $|q_j|\rightarrow \infty$, and then a converging subsequence from $e_j=\frac{q_j}{|q_j|}$, $e_j\rightarrow e_\infty$. We have 
\begin{equation*}-\Delta w_j-(p-2) \left\langle  D^2w_j\frac{D w_j|q_j|^{-1}+e_j}{\abs{D w_j|q_j|^{-1}+e_j}}, \frac{D w_j|q_j|^{-1}+e_j}{\abs{D w_j|q_j|^{-1}+e_j}}\right\rangle=f_j.
\end{equation*}
We show that the uniform limit $w_\infty$ (up to a subsequence) satisfies in the viscosity sense 
\begin{equation}\label{app2}
-\Delta w_{\infty}-(p-2) \left\langle D^2w_{\infty}\,e_{\infty}, e_{\infty}\right\rangle=0\qquad\text{in}\quad B_1,
\end{equation}
with $|e_{\infty}|=1$.

We only show that $w_{\infty}$ is a subsolution of \eqref{app2} (the case of supersolution is similar).
We  fix $\phi\in C^2(\Omega)$ such that $w_{\infty}-\phi$  has a strict maximum at $x_0$. By the uniform convergence of $w_j$ to $w_{\infty}$, there are points $x_j$ such that 
$w_j-\phi$ has a maximum at $x_j$ and $x_j\to x_0$. Since $D\phi (x_j)\to D \phi(x_0)$ and $|q_j|\to \infty$, we know that 
\[
\dfrac{D\phi(x_j)}{|q_j|}\neq -e_j 
\]
for $j$ large. Denoting $A_j:=D \phi(x_j)|q_j|^{-1}$ for short, we get at those points
$$-\Delta \phi(x_j)-(p-2) \left\langle  D^2\phi(x_j)\frac{A_j+e_j}{\abs{A_j+e_j}}, \frac{A_j+e_j}{\abs{A_j+e_j}}\right\rangle\leq f_j(x_j).$$
Since $A_j\rightarrow 0$, we get the desired result.\\

\section{Convergence in the weak formulation}\label{unif limit function}
Assume that $p>2$, $q>\max(2,n, p/2)$, $f_\eps,f\in C(\Om)\cap L^q(\Om)$ and $f_\eps \rightarrow f$ in $L^q(\Om)$. We show that if $u_\eps$ is a weak solution to
\[
-\Delta_p u_\eps=|Du_\eps|^{p-2}f_\eps,
\] 
and if $u_\eps\rightarrow u$ in $C^{1,\a}(K)$ for any $K\subset \subset \Om$, then $u$ is a weak solution to
\[
-\Delta_p u=|Du|^{p-2}f.
\]
For any test function $\phi\in C^\infty_0 (\Omega)$, $u_\eps$ satisfies
\[
\int_\Om |D u_\eps|^{p-2}D u_\eps\cdot D \phi\,dx=\int_\Om |D u_\eps|^{p-2}f_\eps \phi\,dx.
\]
Since $D u_\eps\rightarrow D u$ locally uniformly, we have for all sufficiently small $\eps$, 
\[
|D u_\eps|^{p-2}|D u_\eps\cdot D \phi|\leq (||D u||_{L^\infty(\text{supp}\, \phi)}+1)^{p-1}|D \phi|\in L^1(\Omega),
\]
so by the dominated convergence theorem,
\[
\int_\Om |D u_\eps|^{p-2}D u_\eps\cdot D \phi\,dx\rightarrow \int_\Om |D u|^{p-2}D u\cdot D \phi\,dx.
\]
It remains to show that
\begin{equation}\label{right-handside}
\int_\Om |D u_\eps|^{p-2} f_\eps \phi\,dx\rightarrow \int_\Om |D u|^{p-2} f \phi\,dx.
\end{equation}
Notice that
\begin{equation}\label{identity}
|D u_\eps|^{p-2} f_\eps \phi=|D u_\eps|^{p-2}(f_\eps -f)\phi+|D u_\eps|^{p-2}f\phi.
\end{equation}
Since $D u_\eps\in L^\infty_{\text{loc}}(\Omega)$, by the dominated convergence and identity \eqref{identity}, \eqref{right-handside} holds.

\section{Convergence in the viscosity sense}\label{viscconvergence}

Assume that $h_\eps\in C(\Om)$ and  let $v_{\eps}$ be a viscosity solution to 
\begin{equation}
-\Delta v_{\eps}-(p-2) \frac{D^2v_{\eps} Dv_{\eps}\cdot Dv_{\eps}}{| D v_{\eps}|^2+ \eps^2} +\lambda v_\eps=h_\eps \quad\text{in}\,\,\Om',
\end{equation}
and assume that $v_\eps\to v$ locally uniformly in $\Om'$ and $h_\eps\to h$ locally uniformly. We prove that  the limit $v$ is  a viscosity solution of \eqref{eigennormpl}. Viscosity solutions to \eqref{eigennormpl} are understood in the following sense
\begin{definition}\label{defweak}
Let $\Om'$ be a bounded domain and $2<p<\infty$. An upper semicontinuous function $v$ is  a viscosity subsolution of \eqref{eigennormpl}
if, for all $x_0\in\Om'$  and $\phi\in C^2(\Om')$ such that $v-\phi$ attains a local maximum at $x_0$ and $v(x_0)=\phi(x_0)$, one has either
\begin{equation*}
-\Delta_p^N \phi(x_0)+\lambda v(x_0)\leq h(x_0) \qquad\qquad\qquad\qquad\qquad\text{if}\,\, D \phi(x_0)\neq 0,
\end{equation*}
or  there exists a vector $\eta\in\R^n$ with $|\eta|\leq 1$ such that
\begin{equation*}
-\Delta\phi(x_0)-(p-2)\langle D^2\phi(x_0) \eta, \eta\rangle+\lambda v(x_0)\leq h(x_0)\quad\text{if}\,\, D\phi(x_0)=0.
\end{equation*}
\end{definition}
\noindent The notion of viscosity supersolution is defined similarly and a function $v$ is a viscosity solution to \eqref{eigennormpl} if and only if it is a sub- and supersolution.

 We only show that $v$ is a viscosity subsolution to \eqref{eigennormpl}. To show that $v$ is a viscosity super-solution, one proceeds similarly.
Let $\phi\in C^2$ be such that $v-\phi$ has a local strict maximum at $x_0$ and $v(x_0)=\phi(x_0)$. Since $v_{\eps}\to v$ locally uniformly, there exists a sequence $x_{\eps}\to x_0$ such that $v_{\eps}-\phi$ has a local maximum at $x_{\eps}$. Since $v_{\eps}$ is a viscosity solution of \eqref{perti}, it follows that
\begin{equation}\label{paer}
-\Delta \phi(x_{\eps})-(p-2) \frac{D^2\phi(x_{\eps}) D\phi(x_{\eps})\cdot D\phi(x_{\eps})}{| D \phi(x_{\eps})|^2+ \eps^2}+\lambda v_\eps(x_\eps) \leq h_\eps(x_\eps).
\end{equation}
First suppose that $D\phi(x_0)\neq 0$, then $D\phi(x_{\eps})\neq 0$ for $\eps$ small enough. Since $h_\eps$ converges to $h$ locally uniformly and $v_\eps$ converges to $v$ locally uniformly, passing to the limit in \eqref{paer}, we get that
\begin{equation*}
-\Delta\phi(x_0)-(p-2) \frac{D^2\phi(x_{0}) D\phi(x_{0})\cdot D\phi(x_{0})}{| D \phi(x_{0})|^2}+\lambda v(x_0)\leq h(x_0).
\end{equation*}
Next suppose that $D\phi(x_0)=0$. Noticing that  $\left|\dfrac{D \phi(x_\eps)}{\sqrt{| D \phi(x_{\eps})|^2+ \eps^2}}\right|\leq 1$, it follows that (up to a subsequence) the sequence $\dfrac{D \phi(x_\eps)}{\sqrt{| D \phi(x_{\eps})|^2+ \eps^2}}$ converges to a vector $\eta\in \R^n$ with $|\eta|\leq 1$. Passing to the limit in \eqref{paer}, we get that, there exists a vector $\eta$ such that
\begin{eqnarray*}
-\Delta\phi(x_0)-(p-2) \langle D^2\phi(x_{0})\eta, \eta\rangle +\lambda v(x_0)\leq h(x_0).
\end{eqnarray*}

\section{Uniqueness of  viscosity solutions to \eqref{eigennormpl}}\label{unieigen}

In this section  we prove the uniqueness of viscosity solutions to \eqref{eigennormpl}, where viscosity solutions of \eqref{eigennormpl} are understood in the sense of Definition \ref{defweak} and $\lambda >0$. Notice that, for $\lambda>0$,  the operator 
$$F(X,\xi,r,x):=-\tr(A(\xi) X)+\lambda r-h(x)$$
where $$
A(\xi):=
\begin{cases}
I+(p-2)\ol \xi\otimes \ol \xi &\text{if}\quad \xi\neq 0\\
I+(p-2)\eta\otimes \eta & \text{for a certain}\, \,\eta, |\eta|\leq 1\quad\text{if}\,\, \xi=0
\end{cases}$$
with $\ol \xi :=\dfrac{\xi}{|\xi|}$
is proper, that is
$$F(X,\xi,s,x)\leq F(Y,\xi,r,x)\quad\text{for}\quad Y\leq X,\quad s\leq r.$$
 Now, let $v_1$ and $v_2$ be two continuous viscosity solutions to \eqref{eigennormpl} in $\Om'$  and such that $v_1=v_2$ on $\partial\Om'$. 	We want to show that $v_1=v_2$. We argue by contradiction.  Without loss of generality, we assume that $v_1-v_2$ reaches a positive maximum at an interior point $x_0 \in\Om'$. For $\eps>0$, the function
 $$\Phi(x,y):=v_1(x)-v_2(y)-\dfrac{|x-y|^4}{4\eps},$$
 reaches a maximum in $\overline\Om'\times\overline\Om'$ at $(x_\eps, y_\eps)$. By classical arguments we have that $x_\eps\in \Om', y_\eps\in \Om'$ for $\eps>0$ small enough and $x_\eps\to x_0, y_\eps\to x_0$. We also observe  that the function $x\mapsto v_1(x)-\left(v_2(y_\eps)+\dfrac{|x -y_\eps|^4}{4\eps}\right)=v_1(x)-\phi_1(x)$ reaches a maximum at $x_\eps$ and $y\mapsto  v_2(y)-\left(v_1(x_\eps)-\dfrac{|x_\eps-y|^4}{4\eps}\right)=v_2(y)-\phi_2(y)$ reaches a minimum at $y_\eps$. From the definition of viscosity sub- and supersolution we obtain the following. If $x_\eps=y_\eps$ then $D^2\phi_1(x_\eps)=D^2\phi_2(y_\eps)=0$  and writing the viscosity inequalities we get that 
 $$\lambda v_1(x_\eps)\leq h(x_\eps),\qquad \lambda v_2(x_\eps)\geq h(x_\eps).$$
  It follows that $\lambda(v_1(x_\eps)-v_2(x_\eps))\leq 0$ and passing to the limit we get that 
 $\lambda(v_1(x_0)-v_2(x_0))\leq 0$, which is a contradiction
 since $\lambda>0$ and $v_1(x_0)-v_2(x_0)>0$.\\
If $x_\eps\neq y_\eps$, then  by the theorem of sums \cite[Theorem 3.2]{crandall1992user}
there are 
$$(\xi_x, X)\in \overline{\mathcal{J}}^{2,+}(v_1(x_\eps)), \quad (\xi_y, Y)\in \overline{\mathcal{J}}^{2,-}(v_2(y_\eps))$$ 
with $X\leq Y$ and $\xi_x=\xi_y=D\phi_1(x_\eps)=D\phi_2(y_\eps)\neq 0$.
Writing the viscosity inequalities, we have
\begin{align*}
-\tr (A(\xi_x) X)+\lambda v_1(x_\eps)\leq h(x_\eps)\\
-\tr (A(\xi_x) Y)+\lambda v_2(y_\eps)\geq h(y_\eps).
\end{align*}
Since $A(\xi_x)=I+(p-2)\ol \xi_x\otimes \ol\xi_x\geq 0$ and $X-Y\leq 0$, subtracting the previous two inequalities, we get that
$$\lambda(v_1(x_\eps)-v_2(y_\eps))\leq h(x_\eps)-h(y_\eps)$$ and passing to the limit  we get a contradiction.

\end{document}